\documentclass{article}
\usepackage{graphicx} 

\usepackage{indentfirst}
\usepackage{hyperref}
\hypersetup{
	colorlinks=true,
	citecolor=magenta,      
	urlcolor=cyan,
}

\usepackage{amsthm}
\usepackage{amsfonts}
\usepackage{amssymb}
\usepackage{amsmath}
\usepackage{amsmath,amscd}
\usepackage{graphicx}
\usepackage[all]{xy}
\usepackage{tikz-cd}
\usepackage{circuitikz}
\usepackage{caption}
\usepackage{subcaption}
\usepackage{commath}
\usepackage{breqn}
\usepackage{comment}
\usepackage[T1]{fontenc}
\usepackage{mwe}
\usepackage{graphicx}
\usepackage{dynkin-diagrams}

\graphicspath{ {./images/} }

\newcommand{\link}[1]{ \mathrm{Link}(#1)}
\newcommand{\cone}[1]{ \mathrm{Cone}(#1)}

\theoremstyle{definition}
\newtheorem{defn}{Definition}[section]

\theoremstyle{theorem}
\newtheorem{thm}[defn]{Theorem}

\theoremstyle{conjecture}

\theoremstyle{lemma}
\newtheorem{lemma}[defn]{Lemma}

\theoremstyle{proposition}

\theoremstyle{claim}

\theoremstyle{corollary}
\newtheorem{cor}[defn]{Corollary}

\theoremstyle{remark}
\newtheorem{rmk}[defn]{Remark}

\theoremstyle{example}

\newtheorem*{rep@theorem}{\rep@title}
\newcommand{\newreptheorem}[2]{%
\newenvironment{rep#1}[1]{%
 \def\rep@title{#2 \ref{##1}}%
 \begin{rep@theorem}}%
 {\end{rep@theorem}}}
\makeatother

\newreptheorem{lemma}{Lemma}

\DeclareMathOperator{\cat}{CAT}

\DeclareMathOperator{\Cay}{Cay}
\DeclareMathOperator{\Corners}{Corners}

\DeclareMathOperator{\diam}{diam}
\DeclareMathOperator{\interior}{int}

\title{Automaticity of non-positively curved $k$-fold triangle groups}

\author{Ana Isakovi\'c}

\usepackage{tikz}
\usetikzlibrary{shapes.geometric}
\usetikzlibrary{calc}
\usetikzlibrary{arrows, automata,
	quotes,
	positioning
}
\usepackage{libertine}

\begin{document}

\maketitle

\begin{abstract}
    We show that non-positively curved $k$-fold triangle groups have finitely many cone types, and hence a regular language of all geodesics. Further, we prove that the language of lexicographically first geodesics is both regular and satisfies the fellow traveller property, giving an automatic structure for this family of groups.
\end{abstract}

\section{Introduction}

In this paper, we look at the class of groups called \textit{k-fold triangle groups.} These groups are important because they are the simplest groups that don’t fit into the cubical paradigm, which has been well understood by the work of Agol, Wise and others. Despite their convenient presentations, $k$-fold triangle groups for $k\geq3$ can behave in unpredictable ways, very differently from triangle groups which they generalise. For this reason, they were explored in \cite{LMW18} and \cite{CCKW21} as candidates for non-residually finite hyperbolic groups.

The main result of this paper is the following theorem, which can be seen as an analogue of the Niblo--Reeves theorem in the cubical setting (see \cite{NR98}) and paves the way for performing explicit computations in $k$-fold triangle groups:

\begin{thm}
    Let $G$ be a non-positively curved $k$-fold triangle group. Then $G$ admits a language which is both regular and satisfies the fellow traveller property. Hence, $G$ is automatic.
    \label{automaticity thm}
\end{thm}

Non-positive curvature plays a crucial part in all of the arguments of the paper. Hyperbolic groups have long been known to be automatic (see \cite{Cannon84}), but it is still unknown if all of non-positively curved groups are. Furthermore, by the work of Leary and Minasyan \cite{LM21}, we know that non-positive curvature doesn't guarantee \textit{biautomaticity}. Coxeter groups were proven to be automatic by Brink and Howlett in 1993 \cite{BH93}, completing the work of Davis and Shapiro from 1991 \cite{DS91}, but they were only shown to be biautomatic in 2022 by Osajda and Przytycki \cite{OP22}. Recently, more biautomatic structures were given in \cite{Santos25}.
Another class, which is in spirit close to $k$-fold triangle groups, are small cancellation groups, known to be biautomatic by \cite{GS91}. 

 The main techniques used in proving Theorem \ref{automaticity thm} are disc diagram techniques and the angle structure setting introduced by Wise in \cite{Wise04}. A crucial step is the following lemma, which relates combinatorial geodesics in a non-positively curved $k$-fold triangle group with the $\cat(0)$ geodesics in its corresponding complex:

 \begin{lemma}
    	Let $G$ be a $k$-fold triangle group and $X$ the corresponding triangle complex. A (combinatorial) geodesic in $\Cay(G,S)$ is of length $n$ $\iff$ any $\cat(0)$-geodesic between (the interiors of) the corresponding faces in $X$ crosses exactly $n$ edges of $X$.

        \label{catacomb}
\end{lemma}

\textbf{Roadmap:} In Section 2, we give some background on the key notions of the paper: $k$-fold triangle groups, automaticity, and cone types. In Section 3, we prove Lemma \ref{catacomb}, relating combinatorial and metric structures of $k$-fold triangle groups. In Section 4, we prove Theorem \ref{cones} about the finiteness of cone types, upon which all the further results are based. Finally, in Section 5, we complete the proof of Theorem \ref{automaticity thm}, by showing that their geodesics satisfy the \textit{fellow traveller property.}\\

\textbf{Acknowledgements}: The author was supported by the Department of Pure Mathematics and Mathematical Statistics of the University of Cambridge through an EPSRC Doctoral Training Partnerships award. The author thanks Professor Henry Wilton for the support and guidance provided throughout the writing of this paper.

\section{Preliminaries}

    \subsection{$k$-fold triangle groups}

    Here we introduce $k$-fold triangle groups and a structure theorem by Stallings which gives a geometric approach to the non-positively curved groups of this class. For more information on $k$-fold triangle groups, see \cite{CCKW21}, and for more information on complexes of groups, see \cite{BH99}.

        \begin{defn}
            A \textit{generalised triangle group} is a fundamental group of a triangle of groups whose face group is trivial (see Figure \ref{gtg}). A \textit{$k$-fold triangle group} is a generalised triangle group whose edge groups are all isomorphic to $\mathbb{Z}/k\mathbb{Z}$, for some $k\in\mathbb{N}$.
            
            \begin{figure}[!ht]
                \centering
                \resizebox{0.3\textwidth}{!}{
          	\begin{tikzpicture}
    			
    			\node[isosceles triangle, isosceles triangle apex angle=60,draw,
    			inner sep=0pt,anchor=lower side,rotate=90,draw=black,
    			line width=0.5pt, minimum height=4cm, fill=gray!2] (triangle) at (1.6,-0.05) {};

    			\coordinate[label=left:$V_1$] (A) at (-0.7,0);
    			\coordinate[label=right:$V_2$] (B) at (3.9,0);
    			\coordinate[label=above:$V_3$] (C) at (1.6,4);
    			
    			\coordinate[label=below:$\mathbb{Z}/k$](c) at ($ (A)!.5!(B) $);
    			\coordinate[label=above left:$\mathbb{Z}/k$](b) at ($ (A)!.5!(C) $);
    			\coordinate[label=above right:$\mathbb{Z}/k$](a) at ($ (B)!.5!(C) $);
    			
    			\coordinate[label=below:$\{1\}$](f) at ($ (C)!.5!(c) $);
    			
    		\end{tikzpicture}
                }
                \caption{A $k$-fold triangle group with vertex groups $V_1$,$V_2$ and $V_3$}
                \label{gtg}
             \end{figure}
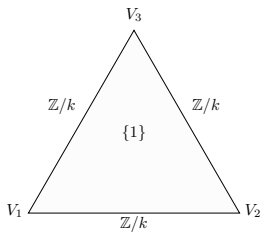
    
	\end{defn}

    In the definition above, when $k=2$ and the vertex groups are dihedral groups, we get a classical triangle group. Furthermore, all the examples where $k=2$ are isomorphic to an amalgamated product of triangle groups and finite groups.

    Under mild hypotheses, a $k$-fold triangle group admits a natural action on a $\cat(0)$ triangle complex with explicit links. To formalise this, we need a notion of a coset graph.

        \begin{defn}
            A \textit{coset graph} of a group $H$ with two subgroups $A$ and $B$ is the bipartite graph $\Gamma_{H(A,B)}$ with vertex set $H/A \cup H/B$ and edge set $H/(A\cap B)$, where the incidence relation is the relation of inclusion.
        \end{defn}

        \begin{thm}[Stallings, \cite{Stallings91}]
            A $k$-fold triangle group $G$ with coset graphs $\Gamma_i:=\Gamma_{V_i(\mathbb{Z}/k,\mathbb{Z}/k)}$ of respective half-girths $r_i$ that satisfy
            $\frac{1}{r_1}+\frac{1}{r_2}+\frac{1}{r_3}\leq 1$
            acts isometrically, by simplicial automorphisms, on a $\cat(0)$ simplicial
            complex $X$ (see Figure \ref{complex X}) of dimension 2. The action has a 2-simplex as a strict fundamental domain, which is isometric to a Euclidean or hyperbolic triangle with angles ($\pi/r_1$, $\pi/r_2$, $\pi/r_3$). The links of the vertices in $X$ are isomorphic to the coset graphs $\Gamma_i$, and the stabilisers of the vertices are isomorphic to the vertex groups $V_i$.
        \end{thm}

    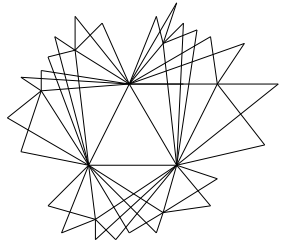
\begin{figure}[!ht]
	\centering
	\resizebox{0.3\textwidth}{!}{%
		\begin{circuitikz}
			\tikzstyle{every node}=[font=\LARGE]
			\draw [short] (8.5,11.5) -- (7,8.5);
			\draw [short] (7,8.5) -- (10.25,8.5);
			\draw [short] (8.5,11.5) -- (10.25,8.5);
			\draw [short] (8.5,11.5) -- (11.75,11.5);
			\draw [short] (11.75,11.5) -- (10.25,8.5);
			\draw [short] (7,8.5) -- (5.25,11.25);
			\draw [short] (5.25,11.25) -- (8.5,11.5);
			\draw [short] (7,8.5) -- (6.5,12.75);
			\draw [short] (6.5,12.75) -- (8.5,11.5);
			\draw [short] (8.5,11.5) -- (9.75,13);
			\draw [short] (9.75,13) -- (10.25,8.5);
			\draw [short] (7,8.5) -- (9.75,6.75);
			\draw [short] (9.75,6.75) -- (10.25,8.5);
			\draw [short] (10.25,8.5) -- (7.25,6.5);
			\draw [short] (7.25,6.5) -- (7,8.5);
			\draw [short] (6.5,12.75) -- (7.5,13.75);
			\draw [short] (7.5,13.75) -- (8.5,11.5);
			\draw [short] (6.5,12.75) -- (6.5,14);
			\draw [short] (6.5,14) -- (8.5,11.5);
			\draw [short] (8.5,11.5) -- (9.5,14);
			\draw [short] (9.5,14) -- (9.75,13);
			\draw [short] (8.5,11.5) -- (10.5,11.5);
			\draw [short] (8.5,11.5) -- (10.25,14.5);
			\draw [short] (10.25,14.5) -- (9.75,13);
			\draw [short] (8.5,11.5) -- (11.5,13.25);
			\draw [short] (11.5,13.25) -- (11.75,11.5);
			\draw [short] (8.5,11.5) -- (12.75,13);
			\draw [short] (12.75,13) -- (11.75,11.5);
			\draw [short] (11.75,11.5) -- (13.5,9.25);
			\draw [short] (13.5,9.25) -- (10.25,8.5);
			\draw [short] (10.25,8.5) -- (14,11.5);
			\draw [short] (14,11.5) -- (11.75,11.5);
			\draw [short] (10.25,8.5) -- (11.5,7);
			\draw [short] (11.5,7) -- (9.75,6.75);
			\draw [short] (10.25,8.5) -- (11.75,8);
			\draw [short] (11.75,8) -- (9.75,6.75);
			\draw [short] (7,8.5) -- (9.5,6);
			\draw [short] (9.5,6) -- (9.75,6.75);
			\draw [short] (7,8.5) -- (8.5,6);
			\draw [short] (8.5,6) -- (9.75,6.75);
			\draw [short] (10.25,8.5) -- (8,5.75);
			\draw [short] (8,5.75) -- (7.25,6.5);
			\draw [short] (7.25,6.5) -- (7.25,5.75);
			\draw [short] (7.25,5.75) -- (10.25,8.5);
			\draw [short] (7,8.5) -- (5.5,7);
			\draw [short] (5.5,7) -- (7.25,6.5);
			\draw [short] (7,8.5) -- (6,6);
			\draw [short] (6,6) -- (7.25,6.5);
			\draw [short] (5.25,11.25) -- (4.5,9);
			\draw [short] (4.5,9) -- (7,8.5);
			\draw [short] (5.25,11.25) -- (4,10.25);
			\draw [short] (7,8.5) -- (4,10.25);
			\draw [short] (5.25,11.25) -- (5.25,12);
			\draw [short] (5.25,12) -- (8.5,11.5);
			\draw [short] (5.25,11.25) -- (4.5,11.75);
			\draw [short] (4.5,11.75) -- (8.5,11.5);
			\draw [short] (7,8.5) -- (5.5,12.5);
			\draw [short] (5.5,12.5) -- (6.5,12.75);
			\draw [short] (7,8.5) -- (5.75,13.25);
			\draw [short] (5.75,13.25) -- (6.5,12.75);
			\draw [short] (9.75,13) -- (10.5,13.75);
			\draw [short] (10.5,13.75) -- (10.25,8.5);
			\draw [short] (9.75,13) -- (11,13.5);
			\draw [short] (11,13.5) -- (10.25,8.5);
		\end{circuitikz}
	}%

        \caption{A part of the complex $X$ for a 3-fold triangle group}
        \label{complex X}
    \end{figure}

    \begin{rmk}
         As $X$ is a $\cat(0)$ complex, every torsion subgroup of $G$ is conjugate into a vertex stabiliser.
         \label{gtg torsion}
    \end{rmk}

    \subsection{Automaticity}

    In this section, we define automaticity, following \cite{Epstein92}.  Informally, a group is considered automatic if all of its elements can be iteratively constructed by a simple recipe, encoded through \textit{finite state automata}, which we are also able to use to check if two words in a specific presentation give the same element or differ by right multiplication by a single generator.
    
\begin{defn}
    A \emph{finite state automaton over S} is a finite directed graph $\Gamma=\Gamma(A,E)$, with a labeling of the edges given by $\phi:E\to\mathcal{P}(S^*)$ and:
    \begin{itemize}
        \item a vertex $a_0 \in A$ called the \emph{start state of  $\Gamma$};
        \item a subset $A_\infty\subseteq A$ called \emph{the accept states of  $\Gamma$}.
    \end{itemize}

We say that a set of words $\Lambda\subseteq S^*$ forms a \emph{regular language} for the finite state automaton $\Gamma$ if it is \emph{accepted} by $\Gamma$, i.e. if for every word $w\in\Lambda$ there is an edge path in $\Gamma$ from the start state $a_0$ to an accept state $a$, whose labels make up $w$.
\end{defn}
We are now ready to give a definition of automaticity for which we use a characterisation from \cite{Epstein92} (see Theorem 2.3.5).

\begin{rmk}
    For a word $w\in\Lambda$, $w=s_1s_2\dots s_n$ we will use $w(i)$ to denote the word up to the $i$-th letter, i.e. $w(i)=s_1s_2\dots s_i$ and $\overline{w}$ to denote the group element represented by $w$.
\end{rmk}

\begin{defn}
            Let $G$ be a finitely generated group.  An \emph{automatic structure for $G$} is a pair $(S,\Lambda)$, where $S$ is a symmetric finite generating set for $G$ and $\Lambda$ is a regular language, such that the following conditions hold:
            \begin{enumerate}
                \item Every element of $G$ is represented by some word in $\Lambda$.
                \item There is a uniform constant $\delta$ such that for every $w,w'\in\Lambda$ with $w'=ws$, for some $s\in S$, $d(\overline{w(i)},\overline{w'(i)})\leq \delta$, for every $i\geq 1$.

            \end{enumerate}
            \label{automaticity_defn}
        \end{defn}
        \begin{rmk}
            Property 2 is often referred to as the \textit{fellow traveller property}. If the same holds not only for words differing by a right multiplication by a single generator, but also by a left multiplication by a single generator, then the group in question is called \textit{biautomatic}.
        \end{rmk}

    \subsection{Cone types}

        Cone types play an important role in the theory of automatic structures. At a group element $g$, the cone type encodes the geodesics from the identity that pass through $g$. Here we give a definition and the main connection to automaticity. For more details see \cite{Calegari11} or \cite{Cannon84}.

        \begin{defn}
            For a group $G$ with a generating set $S$, we define the \textit{cone type} of an element $g\in G$  as: $$\cone{g}=\{h\in G|\text{ there is a geodesic in $\Cay(G,S)$ from id to } gh \text{ which passes through } g\}.$$
            The cone in the sense above is always centered at $e$. We can also looked at the non-centered notion, which we will just call the \textit{cone} of $g$:
            $$\overline{\cone{g}}=\{h\in G|\text{ there is a geodesic in $\Cay(G,S)$ from id to } h \text{ which passes through } g\}.$$
        \end{defn}
        \begin{rmk}
        One can easily see that $g \cone{g}=\overline{\cone{g}}.$
        \end{rmk}
        \begin{thm}[Cannon \cite{Cannon84}]
            In a finitely generated group $G$ with finitely many cone types, the language of all geodesics is regular. Furthermore, if $S$ is a finite, symmetric generating set for $G$ with a fixed total order $\leq$, then the language of lexicographically first (with respect to the order $\leq$) geodesics is also regular.

            \label{Cannon}
        \end{thm}

\section{Relating combinatorial and $\cat(0)$ geodesics}

     The main step in understanding cone types and proving Theorem \ref{automaticity thm} is relating the geodesics in the Cayley graph (combinatorial geodesics) to the geodesics in the cell complex ($\cat(0)$ geodesics). This relies on the combinatorial Gauss--Bonnet formula discussed below. The formula was first proven in \cite{BB96}, and later observed in \cite{MW02}. Here, we use the angle structure setting introduced by Wise in \cite{Wise04}.

    \subsection{Combinatorial Gauss--Bonnet theorem and angle structures}

        \begin{defn}
            Let $Y$ be a 2-complex with angles, i.e. non-negative real numbers $\angle c$ attached to every corner $c$ of a 2-cell in $Y$. We denote by $|\partial f|$ the length of the attaching map of a 2-cell $f$ in $Y$.
            The curvature at a vertex $v\in Y^{(0)}$ is:
            \begin{equation}
                \kappa(v) = 2\pi-\pi \chi(\link{v})-\left(\sum_{c\in \Corners(v)} \angle c\right).
            \end{equation}
            The curvature at a face $f\in Y^{(2)}$ is:
            \begin{equation}
                \kappa(f) = \left( \sum_{c\in \Corners(v)} \angle c \right)-(|\partial f|-2)\pi \,.
            \end{equation}
        \end{defn}
        
        \begin{thm}[Combinatorial Gauss--Bonnet]
            Let $Y$ be a finite 2-complex with angles. With the notation as above:
        
            \begin{equation}
                \sum_{v\in Y^{(0)}} \kappa(v) + \sum_{f\in Y^{(2)}} \kappa(f) = 2\pi \chi(Y).
            \end{equation}
            \label{CGB}
        \end{thm}

    To be able to use the tools above for a $k$-fold triangle group $G$, we need to introduce a suitable 2-dimensional cell complex $C$.
    
    Firstly, let $X$ be a $\cat(0)$ cell complex associated to a $k$-fold triangle group $G$, $k\geq2$. Denote the vertex groups of $G$ by $V_1=\langle a,b \rangle$, $V_2=\langle b,c\rangle$ and $V_3=\langle c,a\rangle$ and the edge groups by $E_1=\langle c\rangle$, $E_2=\langle a\rangle$, and $E_3=\langle b\rangle$. For simplicity of the argument, we redefine the metric on $X$ to be such that each triangle is a Euclidean equilateral triangle, which does not change any important properties of X.

    Furthermore, let $\Cay(G,S)$ be the Cayley graph of $G$ with respect to the generating set $S=\{a,\dots ,({k-1})a,b,\dots,(k-1)b,c,\dots,(k-1)c\}$. Although $\Cay(G,S)$ is not embedded in $X$ (because of torsion), $G$ still acts transitively on the faces of $X$, so there is a correspondence between the faces of $X$ and the elements of $G$, i.e. the vertices in $\Cay(G,S)$. Furthermore, there is a map $i: \Cay(G,S)\to X$, which injectively maps vertices to the centres of triangles in $X$ and maps edges of the Cayley graph to the geodesic segments connecting the relevant centres. This correspondence also allows us to define a \textit{cone type of $t_g\in X^{(2)}$}, the face of $X$ corresponding to an element $g\in G$:
    $$\cone{t_g}=\{t_h\in X^{(2)}|h\in\cone{g}\},$$ where h is the element of $G$ corresponding to $t_h$.

    For $x\in X^{(0)}$, denote by $L^*(x)$ the subgraph of $\Cay(G,S)$ dual to the link $\link{x}$.

   We can now define $C$ as a 2-dimensional cell complex whose 1-skeleton is $\Cay(G,S)$. A 2-cell is attached along every loop in every $L^*(x)$ in $\Cay(G,S)$, which is immersed in $X$. Also, for every edge stabiliser $\mathbb{Z}/k\mathbb{Z}$ of $G$, and every $a,b\in \mathbb{Z}/k\mathbb{Z}$, a triangle is attached with vertices $-a,-b, a+b$.
   
   The metric on $X$ induces the angle structure on $C$. By construction, the triangular cells will have all angles equal to 0, whereas the other cells will have all angles equal to $\frac{2\pi}{3}$.
        
    \subsection{Relating the combinatorial and the metric structure of $G$}

    In this section we prove Lemma \ref{catacomb}. It provides the key geometric insight, by redefining the distance function on $G$ as the number of edges crossed by a $\cat(0)$ geodesic in $X$:
    
     \begin{replemma}{catacomb}
    	Let $G$ be a $k$-fold triangle group and $X$ the corresponding triangle complex. A (combinatorial) geodesic in $\Cay(G,S)$ is of length $n$ $\iff$ any $\cat(0)$-geodesic between the interiors of the corresponding faces in $X$ crosses exactly $n$ edges of $X$.
    \end{replemma}

    \begin{rmk}
        In the case that a $\cat(0)$ geodesic $\gamma$ passes through a vertex $x\in X^{(0)}$ we will look at that as $\gamma$ crossing the smallest sequence of triangles between the one where $\gamma$ enters $x$ and the one where $\gamma$ exits $x$. Similarly, if $\gamma$ passes through a path $p=(p_0,...,p_n)\in X^{(1)}$ we will look at that as $\gamma$ crossing the smallest sequence of triangles between the one where $\gamma$ enters $p_0$ and the one where $\gamma$ exits $p_n$. Such a sequence of triangles is not necessarily unique.

        \label{rmk induced geod}
    \end{rmk}  

Because the triangles of $X$ correspond to vertices in $\Cay(G,S)$, any $\cat(0)$-geodesic $\gamma$ in $X$ induces a path $g_\gamma$ in $\Cay(G,S)$, where crossing an edge between two triangles corresponds to traversing the suitable edge of the Cayley graph. In the case that $\gamma$ passes through a path in $X^{(1)}$, the induced combinatorial geodesic will follow the sequence of triangles from the remark above.

The proof of the lemma relies on showing that the induced path $g_\gamma$ must be a combinatorial geodesic. This is done by looking at a minimal disc diagram bounded by $g_\gamma$ and a combinatorial geodesic.

We now introduce some notions that will help us analyse disk diagrams appearing in the rest of the paper.

\begin{defn}
    Let $\mathcal{D}$ be a disk diagram and $g\subset\partial\mathcal{D}$ a part of its boundary. If a cell $c$ in $\mathcal{D}$ intersects $g$ in at least one edge, we call it a \emph{$g$-cell}. If two $g$-cells $c_1$ and $c_2$ have a non-empty intersection on $g$ we call them \emph{$g$-adjacent}.
\end{defn}

\begin{defn}
    Let $\mathcal{D}$ be a disc diagram and $g$ a path in $\partial\mathcal{D}$. Then the \textit{vertex curvature of $g$ in $\mathcal{D}$} is $$\kappa_{\mathcal{D}}^v(g)=\sum_{v\in \mathcal{D}^{(0)}\cap \interior(g)}\kappa(v)$$ and the \textit{face curvature of $g$ in $\mathcal{D}$} is $$\kappa_{\mathcal{D}}^f(g)=\sum_{\substack{f\in X^{(2)}\\ f \text{ is a $g$-cell}}}\kappa(f).$$ 
    In the special case when the faces intersecting $g$ are triangles we define the \textit{triangle curvature of $g$ in $\mathcal{D}$} as $$\kappa_{\mathcal{D}}^t(g)=\sum_{\substack{t\in X^{(2)}\\ t \text{ is a triangular $g$-cell}}}\kappa(f).$$ 
    Furthermore, if there is another path $g$ in $\partial\mathcal{D}$, $$\kappa_{\mathcal{D}}^f(g,g')=\sum_{\substack{f\in X^{(2)}\\ f\text{ is a $g$-cell}\\f \text{ is a $g'$-cell}}} \kappa(f).$$
\end{defn}

Before we move on to the proof of the lemma, we state and prove one of its ingredients, which gives more insight into paths induced by geometric geodesics:

\begin{lemma}
    Let $X$, $C$, $\gamma$ and $g_{\gamma}$ be as above and let $\mathcal{D}$ be a minimal disc diagram in $C$ such that $g_\gamma \subset \partial\mathcal{D}.$ Then $$\kappa_{\mathcal{D}}^v(g_\gamma)+\frac{1}{2}\kappa_{\mathcal{D}}^f(g_\gamma)+\frac{1}{2}\kappa_{\mathcal{D}}^t(g_\gamma)\leq \frac{2\pi}{3}.$$ 
        \label{induced geodesic}
\end{lemma}
 \begin{proof}
            We will compare the direction of $g_{\gamma}$, to the direction of $\gamma$ (see Figures \ref{figure 3a}, \ref{figure 3b}).

            Take a sequence of triangles $t_i\in X, i\in\{0,\dots,n\}$ which correspond to vertices $p_i\in X, i\in\{0,\dots,n\}$ of $g_\gamma$ from Remark \ref{rmk induced geod}. Moreover, we pick a leftmost such sequence (with respect to the direction of $\gamma$). Next, take a subsequence $t_{i_j}\in X, j\in\{0,\dots,m\}$ such that $t_{i_j}\cap\gamma\neq\{pt\}$. We will denote the subset of indices $\{i_{j}\}$ by $I_s$. Since $\gamma$ starts and ends at interiors of triangles, we have that $i_0=0$ and $i_m=n$.

            We start with the definition of vertex curvature over $g_\gamma$:
        
            \begin{equation}
                \kappa_{\mathcal{D}}^v(g_{\gamma})=\sum_{i=1}^{n-1}\kappa(p_i)=\sum_{i=1}^{n-1}\left(\pi-\sum_{c\in \Corners(p_i)}\angle c\right).
                \label{g gamma curvature}
            \end{equation}
        
            We assume that $g_\gamma$ bounds the disk diagram $\mathcal{D}$ from the right, i.e. that $\mathcal{D}$ is on the left of $g_\gamma$. The sum of the corner angles at $p_i$ is bounded below by the angle between the segments $[p_{i-1},p_i]$ and $[p_i,p_{i+1}]$. Let $g_\gamma(p_{i})$ be the unit vector corresponding to the segment $[p_i,p_{i+1}]$ at $p_i$. Because $[p_i,p_{i+1}]$ can be seen in $X$ inside two adjacent triangles, so inside a part of a Euclidean plane, the vector corresponding to $[p_{i+1},p_i]$ at $p_{i+1}$ is $-g_\gamma(p_{i})$. 
            Because all the triangles in $X$ are equilateral, $\angle(g_\gamma(p_{i}),-g_\gamma(p_{i-1}))\in\{0,\frac{2\pi}{3}, \frac{4\pi}{3}\}$. Furthermore, angle zero would imply that $\gamma$ enters and exits some triangle via the same edge, which would contradict the convexity of triangles. So:

            \begin{equation}
                \angle(g_\gamma(p_{i}),-g_\gamma(p_{i-1}))\in\{\frac{2\pi}{3}, \frac{4\pi}{3}\}
                \label{angle range}.
            \end{equation}
            
            From equation \ref{g gamma curvature} we get the following inequality:
        
            \begin{equation}
               \kappa_{\mathcal{D}}^v(g_{\gamma})
                \leq\sum_{i=1}^{n-1}\pi-\angle(g_\gamma(p_{i}),-g_\gamma(p_{i-1})).
                \label{g gamma curvature 1}
            \end{equation}
        
            For each triangle in the subsequence $\{t_{i_j}\}$, $\interior(\gamma\cap t_{i_j})\neq\emptyset$, so for each $j$, we can pick a point $q_j\in\interior(\gamma\cap t_{i_j})$ and define $\gamma(p_i)$ to be the unit vector parallel to the tangent vector of $\gamma$ at $q_i$. Because the triangles are Euclidean, this is well-defined. We want to compare the directions of different vectors. For this purpose we fix an orientation on the triangle containing $p_0$ and propagate it to all the other triangles in the sequence $\{t_i\}$.
        
            Because $\gamma$ and $g_\gamma$ always cross the same edges of the underlying triangles $t_{i_j}$ in $X$, the angle between them is always $\leq \pi/2$. We separate the indices $I_s$ into two sets, depending on which side of $g_\gamma(p_{i_j})$ the vector $\gamma(p_{i_j})$ is. In the first case, $\gamma(p_i)$ is on the left and $\angle(g_\gamma(p_{i}),\gamma(p_{i}))\in[0,\frac{\pi}{2}]$. We denote the set of such indices by $I_l$. In the second case, $\gamma(p_i)$ is on the right, and $\angle(g_\gamma(p_{i}),\gamma(p_{i}))\in[\frac{3\pi}{2},2\pi)$. This set of indices will be denoted by $I_r$. Hence we can rewrite \ref{g gamma curvature 1} as:
            \begin{align}
                \kappa_{\mathcal{D}}^v(g_{\gamma})
                \leq&\sum_{i\in \{1,\dots,n-1\}\setminus I_s}\pi-\angle(g_\gamma(p_{i}),-g_\gamma(p_{i-1}))+\sum_{i\in I_l}\pi-\angle(g_\gamma(p_{i}),-g_\gamma(p_{i-1}))\nonumber\\
                &+\sum_{i\in I_r}\pi-\angle(g_\gamma(p_{i}),-g_\gamma(p_{i-1})). \label{g gamma curvature 2}
            \end{align}
        
            In the case of $I_l$, because (by \ref{angle range}) $\angle(g_\gamma(p_i) ,-g_\gamma(p_{i-1}))\geq \frac{2\pi}{3}$, we have that $\gamma(p_i)$ is between $g_\gamma(p_i)$ and $-g_\gamma(p_{i-1})$, so  
            \begin{equation}
                \angle(g_\gamma(p_{i}),-g_\gamma(p_{i-1}))=\angle(g_\gamma(p_{i}),\gamma(p_{i}))+\angle(\gamma(p_i),-g_\gamma(p_{i-1})).
            \end{equation}
        
            In the case of $I_r$, because (by \ref{angle range}) $\angle(g_\gamma(p_i) ,-g_\gamma(p_{i-1}))\leq \frac{4\pi}{3}$, we have that $\gamma(p_i)$ is between $-g_\gamma(p_{i-1})$ and $g_\gamma(p_i)$, so 
            \begin{equation}
               \angle(g_\gamma(p_{i}),-g_\gamma(p_{i-1}))=\angle(g_\gamma(p_{i}),\gamma(p_{i}))+\angle(\gamma(p_i),-g_\gamma(p_{i-1}))-2\pi. 
            \end{equation}
        
            This allows us to rewrite the bound \ref{g gamma curvature 2} for the curvature of $g_{\gamma}$ as:
            \begin{align}
                \kappa_{\mathcal{D}}^v(g_\gamma)
                    \leq
                    &\sum_{i\in \{1,\dots,n-1\}\setminus I_s}\pi-\angle(g_\gamma(p_{i}),-g_\gamma(p_{i-1}))\nonumber\\
                    &+\sum_{i\in I_l}\pi-(\angle(g_\gamma(p_{i}),\gamma(p_{i}))+\angle(\gamma(p_i),-g_\gamma(p_{i-1}))) \nonumber\\
                    +&\sum_{i\in I_r}\pi-(\angle(g_\gamma(p_{i}),\gamma(p_{i}))+\angle(\gamma(p_i),-g_\gamma(p_{i-1}))-2\pi).
                    \label{g gamma curvature 3}
            \end{align}
            
            We will now use the fact that $\gamma$ is a geodesic. If two consecutive triangles with indices in $I_s$, $t_{i_{j-1}}$ and $t_{i_{j}}$ share an edge, then $\gamma(p_{i_{j-1}})$ is parallel to $\gamma(p_{i_{j}})$ in a locally Euclidean space, so we can regard them as the same vector.
            In the case when $t_{i_{j-1}}$ and $t_{i_{j}}$ only share a vertex $v_j$, denote the sequence of triangles between them, $t_{i_{j-1}+1},\dots,t_{i_j-1}$, by $\mathcal{T}_j$ and its length by $l_j$. The sequence of triangles now forms a non-positively curved space, with negative curvature concentrated at $v_j$, through which $\gamma$ passes. Hence $\gamma(p_{i_{j-1}})$ is not necessarily parallel to $\gamma(p_{i_{j}})$ along $g_{\gamma}$. However, we still have all the needed information about the change of angles. Namely, we can show how the length of $\mathcal{T}_j$, $l_j$, affects the relevant vectors in triangles $t_{i_{j-1}}$ and $t_{i_{j}}$, which will help us simplify inequality \ref{g gamma curvature 3}.
            
            We know that $l_j> 0$, because $\gamma$ is a geodesic. If $l_j=1$, then $\gamma$ is contained in the edges of $t_{i_{j-1}}$ and $t_{i_{j}}$ not intersecting the one triangle between them, $t_{i_{j-1}+1}$. Since we know the exact positions of all the relevant vectors, we also know the following angles:
            \begin{align}
                \angle(\gamma(p_{i_j}),-g_\gamma(p_{{i_j}-1}))
                =\begin{cases}&\frac{5\pi}{6},\hspace{1.1cm} \text{if } v_j\notin\mathcal{D}\\
                &\frac{7\pi}{6},\hspace{1.1cm} \text{if } v_j\in\mathcal{D}\
                \end{cases} \text{ , and} 
                \label{vector rel 1.1}
            \end{align}
            \begin{align}
                \angle(\gamma(p_{i_{j-1}}),-g_\gamma(p_{i_{j-1}}))
                =\begin{cases}&\frac{7\pi}{6},\hspace{1.1cm} \text{if } v_j\notin\mathcal{D}\\
                &\frac{5\pi}{6},\hspace{1.2cm} \text{if } v_j\in\mathcal{D}\
                \end{cases} \text{ .} 
                \label{vector rel 1.2}
            \end{align}
            Hence, we get the following relationship between them:
            \begin{align}
                \angle(\gamma(p_{i_j}),-g_\gamma(p_{{i_j}-1}))=\begin{cases}&\angle(\gamma(p_{i_{j-1}}),-g_\gamma(p_{i_{j-1}}))-\frac{\pi}{3},\hspace{0.5cm} \text{if } v_j\notin\mathcal{D}\\
                &\angle(\gamma(p_{i_{j-1}}),-g_\gamma(p_{i_{j-1}}))+\frac{\pi}{3},\hspace{0.5cm} \text{if } v_j\in\mathcal{D}\
                \end{cases}\text{ .}
                \label{angle rel 1}
            \end{align}
            We observe that $\pm\frac{\pi}{3}$ in the equation above corresponds to the maximal curvature at $p_{i_{j-1}+1}$, and hence:
            \begin{align}
                -\angle(\gamma(p_{i_j}),-g_\gamma(p_{{i_j}-1}))+\sum_{i_{j-1}<i<i_j}\kappa(p_i)\leq-\angle(\gamma(p_{i_{j-1}}),-g_\gamma(p_{i_{j-1}})).
                \label{angle replacement 1}
            \end{align}
            
            Now assume that $l_j=2$. The geodesic $\gamma$ seen inside the space formed by triangles $t_{i_{j-1}},\dots ,t_{i_j}$ forms an angle greater than or equal to $\pi$ at the point $v_j$. If the angle is equal to $\pi$, then $v_j$ has zero curvature, and we get the following relationship between angles:
            \begin{align}
                \angle(\gamma(p_{i_j}),-g_\gamma(p_{{i_j}-1}))=\begin{cases}&\angle(\gamma(p_{i_{j-1}}),-g_\gamma(p_{i_{j-1}}))-\frac{2\pi}{3},\hspace{0.5cm} \text{if } v_j\notin\mathcal{D}\\
                &\angle(\gamma(p_{i_{j-1}}),-g_\gamma(p_{i_{j-1}}))+\frac{2\pi}{3},\hspace{0.5cm} \text{if } v_j\in\mathcal{D}\
                \end{cases}
                \label{angle rel 2.1}
            \end{align}
            As before, we observe that $\pm\frac{2\pi}{3}$ in the formula above is an upper bound on the curvature of the two vertices between $p_{i_{j-1}}$ and $p_{i_j}$ on the path of length 3, which connects them, so:
            \begin{align}
                -\angle(\gamma(p_{i_j}),-g_\gamma(p_{{i_j}-1}))+\sum_{i_{j-1}<i<i_j}\kappa(p_i)\leq-\angle(\gamma(p_{i_{j-1}}),-g_\gamma(p_{i_{j-1}})).
                \label{angle replacement 2.1}
            \end{align}

            Now assume that $l_j$ is still 2, but the angle of $\gamma$ at $v_j$ is strictly greater than $\pi$. Then $v_j$ has negative curvature, and the angle difference is in between the previous two cases. Namely, we get the following relationship between angles:
            \begin{align}
                \angle(\gamma(p_{i_j}),-g_\gamma(p_{{i_j}-1}))\geq\begin{cases}&\angle(\gamma(p_{i_{j-1}}),-g_\gamma(p_{i_{j-1}}))-\frac{2\pi}{3},\hspace{0.5cm} \text{if } v_j\notin\mathcal{D}\\
                &\angle(\gamma(p_{i_{j-1}}),-g_\gamma(p_{i_{j-1}}))+\frac{\pi}{3},\hspace{0.5cm} \text{if } v_j\in\mathcal{D}\
                \end{cases}
                \label{angle rel 2.2}
            \end{align}
            If $v_j\notin \mathcal{D}$, then, as before, $-\frac{2\pi}{3}$ in the formula above is an upper bound on the curvature of the two vertices between $p_{i_{j-1}}$ and $p_{i_j}$ on the path of length 3, which connects them. On the other hand, if $v_j\in \mathcal{D}$, then both $p_{i_{j-1}+1}$ and $p_{i_{j-1}+2}$ can be positively curved, but we only get $-\frac{\pi}{3}$ curvature from the formula above for offsetting the vertex curvature. However, because $v_j$ is negatively curved, there is a cell $c_j$ in $\mathcal{D}$ of girth at least 8 containing vertices $p_{i_{j-1}+1}$ and $p_{i_{j-1}+2}$, which contributes the other $-\frac{\pi}{3}$ for offsetting $\sum_{i_{j-1}<i<i_j}\kappa(p_i)$. Hence:
            \begin{align}
                -\angle(\gamma(p_{i_j}),-g_\gamma(p_{{i_j}-1}))+\sum_{i_{j-1}<i<i_j}\kappa(p_i)+\frac{1}{2}\kappa(c_j)\leq-\angle(\gamma(p_{i_{j-1}}),-g_\gamma(p_{i_{j-1}})).
                \label{angle replacement 2.2}
            \end{align}
            
            Finally, when $l_j\geq3$, vectors $\gamma(p_{i_j})$ and $\gamma(p_{i_j-1})$ independently have an angle $\frac{\pi}{3}$ of freedom in their respective triangles. Since $\gamma$ and $g_\gamma$ always enter and exit the triangles of $X$ through the same edge, we get the following estimates:
            \begin{align}
                \angle(\gamma(p_{i_j}),-g_\gamma(p_{{i_j}-1}))
                \in\begin{cases}&[\frac{\pi}{2},\frac{5\pi}{6}],\hspace{1.1cm} \text{if } v_j\notin\mathcal{D}\\
                &[\frac{7\pi}{6},\frac{3\pi}{2}],\hspace{0.9cm} \text{if } v_j\in\mathcal{D}\
                \end{cases} \text{ , and} 
                \label{vector rel3.1}
            \end{align}
            \begin{align}
                \angle(\gamma(p_{i_{j-1}}),-g_\gamma(p_{i_{j-1}}))
                \in\begin{cases}&[\frac{7\pi}{6},\frac{3\pi}{2}],\hspace{1.1cm} \text{if } v_j\notin\mathcal{D}\\
                &[\frac{\pi}{2},\frac{5\pi}{6}],\hspace{1.2cm} \text{if } v_j\in\mathcal{D}\
                \end{cases} \text{ .} 
                \label{vector rel3.2}
            \end{align}
            Consequently:
            \begin{equation}
                -\angle(\gamma(p_{i_j}),-g_\gamma(p_{{i_j}-1}))\leq\begin{cases}
                &-\angle(\gamma(p_{i_{j-1}}),-g_\gamma(p_{i_{j-1}}))+\pi,\hspace{1cm} \text{if } v_j\notin\mathcal{D}\\
                &-\angle(\gamma(p_{i_{j-1}}),-g_\gamma(p_{i_{j-1}}))-\frac{\pi}{3},\hspace{0.9cm} \text{if } v_j\in\mathcal{D}.
            \end{cases}
            \label{angle rel 3}
            \end{equation}
            When $v_j\notin\mathcal{D}$ the vertices of the path around $v_j$ are negatively curved and there are $l_j\geq3$ of them, so $\sum_{i_{j-1}<i<i_j}\kappa(p_i)\leq-\pi$ and hence:
            \begin{equation}
                 -\angle(\gamma(p_{i_j}),-g_\gamma(p_{{i_j}-1}))+\sum_{i_{j-1}<i<i_j}\kappa(p_i)\leq
                -\angle(\gamma(p_{i_{j-1}}),-g_\gamma(p_{i_{j-1}})).
                \label{angle replacement 3.1}
            \end{equation}
            On the other hand, when $v_j\in\mathcal{D}$, since $\mathcal{D}$ is reduced, the $l_j$ vertices around $v_j$ are positively curved and belong to the same cell, $c_j\in X^{(2)}$. This cell may contain 0,1 or 2 $g_\gamma$-adjacent triangles.
            
            Since the path of triangles chosen in the beginning was the leftmost one, the girth of the cell $c_{j}$ has to be strictly larger than $2l_j$. Moreover, if $c_j$ has no $g_\gamma$-adjacent triangles, the girth of it has to be strictly larger than $2l_j+2$. So, the curvature of $c_j$ is:
            \begin{align}
                \kappa(c_{j_k})\leq \begin{cases}&-(2l_j+2-6)\frac{\pi}{3},\hspace{1.1cm} \text{if $c_j$ has a $g_\gamma$-adjacent triangle}\\
                &-(2l_j+4-6)\frac{\pi}{3},\hspace{1.1cm} \text{if $c_j$ has no $g_\gamma$-adjacent triangles}\
                \end{cases} \text{ .} 
                \label{cell curv approx}
            \end{align}
            This further means that:
            \begin{align}
                \frac{1}{2}\kappa(c_{j_k})\leq \begin{cases}&-(l_{j_k}-2)\frac{\pi}{3},\hspace{1.1cm} \text{if $c_j$ has a $g_\gamma$-adjacent triangle}\\
                &-(l_{j_k}-1)\frac{\pi}{3},\hspace{1.1cm} \text{if $c_j$ has no $g_\gamma$-adjacent triangles}\
                \end{cases} \text{ .} 
                \label{cell curv approx 2}
            \end{align}
            When $c_j$ has no $g_\gamma$-adjecent triangles, the cell curvature from the formula above covers $l_j-1$ positively curved vertices, and the final one is covered by the angle difference in formula \ref{angle rel 3}.
            If $c_j$ has a $g_\gamma$-adjecent triangle, then the cell curvature from the formula above covers $l_j-2$ positively curved vertices; one more is covered by the angle difference in formula \ref{angle rel 3}, and the final one is covered by half the curvature of the triangle, $t_j$. Note that we only use half as two cells can share a $g$-adjacent triangle. Finally, we get:
            \begin{align}
                -\angle(\gamma(p_{i_j}),-g_\gamma(p_{{i_j}-1}))+\sum_{i_{j-1}<i<i_j}\kappa(p_i)+\frac{1}{2}\kappa(c_{j})+\kappa(t_j) \leq
                -\angle(\gamma(p_{i_{j-1}}),-g_\gamma(p_{i_{j-1}}))
                \label{angle replacement 3.2}
            \end{align}

            Finally, we incorporate estimates \ref{angle replacement 1}, \ref{angle replacement 2.1}, \ref{angle replacement 2.2}, \ref{angle replacement 3.1} and \ref{angle replacement 3.2} into \ref{g gamma curvature 3} and add the face curvature (where all the cells we don't need to offset the vertex curvature can be ignored, as all the cells of $\mathcal{D}$ are non-positively curved). This gives us:
            \begin{align}
                \kappa_{\mathcal{D}}^v(g_\gamma)+\frac{1}{2}\kappa_{\mathcal{D}}^f(g_\gamma) + \frac{1}{2}\kappa_{\mathcal{D}}^t(g_\gamma)
                    \leq -&\angle(\gamma(p_0),-g_\gamma(p_0)) + \angle(\gamma(p_{i_{m-1}}),-g_\gamma(p_{i_{m-1}})) \nonumber\\
                    +&\sum_{i\in I_l}\pi-(\angle(g_\gamma(p_{i}),\gamma(p_{i}))+\angle(\gamma(p_i),-g_\gamma(p_{i}))) \nonumber\\
                    +&\sum_{i\in I_r}\pi-(\angle(g_\gamma(p_{i}),\gamma(p_{i}))+\angle(\gamma(p_i),-g_\gamma(p_{i}))-2\pi). 
            \end{align}
        
            Now, in the first case, we have that $\gamma(p_i)$ is between $g_\gamma(p_i)$ and $-g_\gamma(p_i)$, so 
            \begin{equation}                                 \angle(g_\gamma(p_{i}),\gamma(p_{i}))+\angle(\gamma(p_i),-g_\gamma(p_{i}))=\angle(g_\gamma(p_{i}),-g_\gamma(p_{i}))=\pi.
            \end{equation}
        
            In the second case, $\gamma(p_i)$ is between $-g_\gamma(p_{i})$ and $g_\gamma(p_i)$, so 
            \begin{equation}                        \angle(g_\gamma(p_{i}),\gamma(p_{i}))+\angle(\gamma(p_i),-g_\gamma(p_{i}))=\angle(g_\gamma(p_{i}),-g_\gamma(p_{i}))+2\pi=\pi+2\pi. 
            \end{equation}
        
            Hence, the only non-zero contribution to the curvature bound comes from the $0$-th and the $(n-1)$-st term:
            \begin{align}
                \kappa_{\mathcal{D}}^v(g_\gamma)+\frac{1}{2}\kappa_{\mathcal{D}}^f(g_\gamma)+\frac{1}{2}\kappa_{\mathcal{D}}^t(g_\gamma)
                    \leq -\angle(\gamma(p_0),-g_\gamma(p_0)) + \angle(\gamma(p_{i_{m-1}}),-g_\gamma(p_{i_{m-1}})).
            \end{align}
        
            In general $\angle(\gamma(p_{i_j}),-g_\gamma(p_{i_j}))\in[\frac{\pi}{2},\frac{3\pi}{2}]$ for $0\leq j\leq m-1$, but since we have that $\gamma$ starts and ends at interiors of triangles, we have sharper bounds for the end cases: $\angle(\gamma(p_0),-g_\gamma(p_0)), \angle(\gamma(p_{i_{m-1}}),-g_\gamma(p_{i_{m-1}}))\in(\frac{\pi}{2},\frac{3\pi}{2})$. So:
        
            \begin{align}
                \kappa_{\mathcal{D}}^v(g_\gamma)+\frac{1}{2}\kappa_{\mathcal{D}}^f(g_\gamma)+\frac{1}{2}\kappa_{\mathcal{D}}^t(g_\gamma)< \pi.
            \end{align}
        
            As the direction of $g_{\gamma}$ can only change for a multiple of $\frac{\pi}{3}$ at a time, we can sharpen the estimate above to get the wanted one:
        
            \begin{align}
                 \kappa_{\mathcal{D}}^v(g_\gamma)+\frac{1}{2}\kappa_{\mathcal{D}}^f(g_\gamma)+\frac{1}{2}\kappa_{\mathcal{D}}^t(g_\gamma)
                    \leq \frac{2\pi}{3}.
            \end{align}

            
            \begin{figure}[!ht]
            \centering
            \begin{subfigure}[t]{0.65\linewidth}
            \resizebox{1\textwidth}{!}
                {%
                \begin{circuitikz}
                \tikzstyle{every node}=[font=\normalsize]
                \draw [short] (3.25,7.25) -- (5,10.5);
                \draw [short] (5,10.5) -- (7,7.25);
                \draw [short] (7,7.25) -- (3.25,7.25);
                \draw [short] (5,10.5) -- (8.75,10.5);
                \draw [short] (8.75,10.5) -- (7,7.25);
                \draw [ color={rgb,255:red,189; green,20; blue,20}, ->, >=Stealth] (5,8.5) -- (7,9.25)node[pos=0.5, fill=white]{$g_{\gamma}$};
                \draw [ color={rgb,255:red,20; green,161; blue,189}, short] (5,9.5) -- (13.5,7)node[pos=0.5, fill=white]{$\gamma$};
                \end{circuitikz}
                }%
                \caption{starting direction of $g_\gamma$ with respect to $\gamma$}
                \label{figure 3a}
            \end{subfigure}
            \hfill 
            \begin{subfigure}[t]{0.27\linewidth}
            \resizebox{1\textwidth}{!}
                {%
                \begin{circuitikz}
                \tikzstyle{every node}=[font=\normalsize]
                \draw [ color={rgb,255:red,20; green,161; blue,189}, ->, >=Stealth] (5.25,8) -- (9.25,8)node[pos=0.5, fill=white]{$\gamma$};
                \draw [ color={rgb,255:red,232; green,206; blue,38}, short] (5.25,11.25) -- (5.25,5);
                \draw [ color={rgb,255:red,189; green,20; blue,20}, ->, >=Stealth] (5.25,8) -- (8.75,10)node[pos=0.5, fill=white]{$g_\gamma$};
                \draw [ color={rgb,255:red,189; green,20; blue,20}, ->, >=Stealth, dashed] (5.25,8) -- (5.5,11.25);
                \draw [ color={rgb,255:red,189; green,20; blue,20}, ->, >=Stealth, dashed] (5.25,8) -- (8.5,5.75);
                \node [font=\normalsize, color={rgb,255:red,189; green,20; blue,20}] at (5.75,8.75) {$\pi/3$};
                \node [font=\normalsize, color={rgb,255:red,189; green,20; blue,20}] at (6,8) {$\pi/3$};
                \end{circuitikz}
                }%
                \caption{all possible directions of $g_\gamma$ with respect to $\gamma$}
                \label{figure 3b}
            \end{subfigure}
           \caption{Comparing the directions of $\gamma$ and $g_{\gamma}$}
            \label{fig:my_label}
            \end{figure}
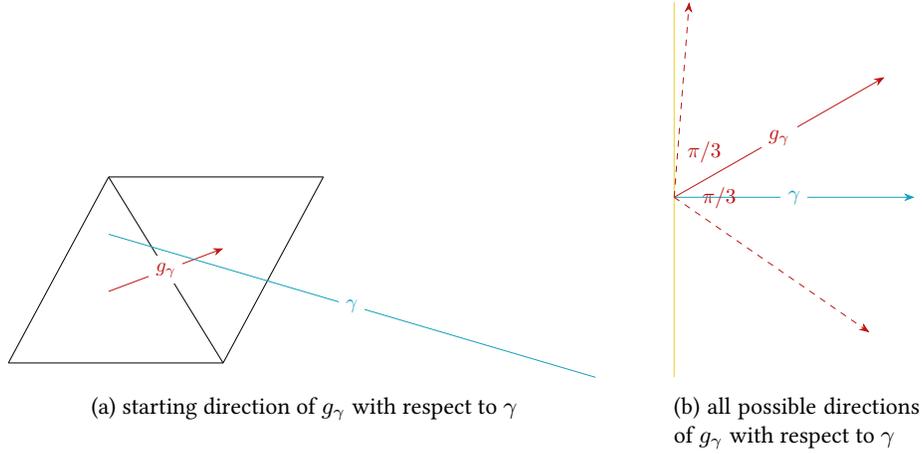
            
\end{proof}

We can now move on to the proof of Lemma \ref{catacomb}.  
 
 \begin{proof}[Proof of Lemma \ref{catacomb}]
      Let $g$ be a geodesic in $\Cay(G,S)$ from a vertex $u$ to a vertex $v$ and let $\gamma$ be the unique geodesic in $X$ from a point $p_u$ in triangle $u$ to a point $p_v$ in triangle $v$. Further, let $g_{\gamma}$ be the path in $\Cay(G,S)$ induced by $\gamma$. 
        
            We want to show that $g_{\gamma}$ is also a geodesic in $\Cay(G,S)$, i.e. that we can geodesically transform $g$ into $g_{\gamma}$. This is done by observing the two combinatorial paths inside the cell complex $C$ and doing induction on the area between them.
            
            Take the first segment $[t_1,t_2]$ such that $g$ and $g_{\gamma}$ only meet at $g(t_1)=g_{\gamma}(t_1)=u'$ and $g(t_2)=g_{\gamma}(t_2)=v'$. On this segment, the paths bound a disc diagram $\mathcal{D}$ over the cell complex $C$. Moreover, we can take the disc diagram to be of minimal area. Then the area between the two paths will be the sum of the areas of minimal disc diagrams over all segments where the paths are disjoint.
        
            We now look at the curvature of the faces and vertices of $\mathcal{D}$.

            By construction of $C$, $\mathcal{D}$ contains two types of cells. The first type are $n$-gons of size $\geq6$, which by the formula for curvature of faces have non-positive curvature $-(n-6)\frac{\pi}{3}$. The second type are triangles, coming from the order $k$-elements. Since their corners have angle zero, the triangles will be negatively curved with curvature $-\pi$.
            
            By minimal area, the interior vertices can also be shown to be of non-positive curvature. Namely, if a vertex $v$ is in the interior of $\mathcal{D}$, then $\link{v}$ is either a loop, which needs to be of length at least $2\pi$, or it is a tree, which means that $\mathcal{D}$ could be further reduced, contradicting the minimal area assumption.
        
            Consider a vertex $v$ on the boundary. It has $$\sum_{c\in \Corners(v)}c=s \frac{2\pi}{3},  s\geq0,$$ because all the interior angles of the complex $C$ are either $\frac{2\pi}{3}$ or 0. We will show that $s\geq 1$. The case $s=0$ is excluded by disjointness at ends, geodesicity of $g$ or underlying geodesic of $g_{\gamma}$. Namely, if $v\in \{u',v'\}$ is such that the sum of the corners at $v$ is zero, it would mean that $g$ and $g_\gamma$ intersect at a whole segment containing $u'$ or $v'$, which was excluded by a restriction above. Further, $v\in g$ with corner angle zero would mean backtracking, contradicting the combinatorial geodesicity of $g$. Finally, we saw in the proof of Lemma \ref{induced geodesic} that $\gamma$ being a $\cat(0)$ geodesic excludes the case of corner angle zero on $g_\gamma$.
            Hence, for a vertex $v$ on the boundary, $$\sum_{c\in \Corners(v)}c=s \frac{2\pi}{3},  s\geq1,$$ which corresponds to $\kappa(v)=\pi-s\frac{2\pi}{3}, s\geq1$. Particularly, the curvature at vertices $u',v'$ is $\kappa(u'),\kappa(v')\leq\frac{\pi}{3}$.
        
            Further, by Lemma \ref{induced geodesic}, \begin{equation}
                \kappa_{\mathcal{D}}^v(g_\gamma)+\frac{1}{2}\kappa_{\mathcal{D}}^f(g_\gamma)+\frac{1}{2}\kappa_{\mathcal{D}}^t(g_\gamma)\leq \frac{2\pi}{3}.
            \end{equation}
            By the combinatorial Gauss--Bonnet theorem for $\mathcal{D}$, we then have the following:
            \begin{equation}
                2\pi \chi(\mathcal{D}) =\sum_{v\in \partial \mathcal{D}^{(0)}} \kappa(v) + \sum_{v\in \interior\mathcal{D}^{(0)}} \kappa(v) - \sum_{f\in \mathcal{D}^{(2)}} \kappa(f).
                \label{gb}
            \end{equation}
            This can further be rewritten as:
            \begin{align}
             2\pi\chi(\mathcal{D}) =\kappa_{\mathcal{D}}^v(g)&+\kappa_{\mathcal{D}}^v(g_{\gamma}) +\kappa(u') + \kappa(v') +\sum_{v\in \interior\mathcal{D}^{(0)}} \kappa(v) \nonumber\\ 
                &+ \kappa^f_{\mathcal{D}}(g)+ \kappa^f_{\mathcal{D}}(g_{\gamma})-\kappa_{\mathcal{D}}^f(g,g_\gamma)+\sum_{\substack{f\in \mathcal{D}^{(2)}\\ f \text{is not a $g$-cell}\\ \text{or a $g_\gamma$-cell}}}\kappa(f).
                \label{gb2}
            \end{align}
            
            We know that $\kappa(u')\leq\frac{\pi}{3}$,$\kappa(v')\leq\frac{\pi}{3}$,  $\chi(\mathcal{D})=1$, and vertices in the interior as well as all the faces have non-positive curvature. Incorporating these estimates into the equation \ref{gb2} gives the following: 
            \begin{equation}
                \kappa_{\mathcal{D}}^v(g)+\kappa_{\mathcal{D}}^v(g_{\gamma})  +\kappa^f_{\mathcal{D}}(g)+ \kappa^f_{\mathcal{D}}(g_{\gamma})-\kappa_{\mathcal{D}}^f(g,g_\gamma)\geq \frac{4\pi}{3}.
                \label{g, g gamma curvature approx}
            \end{equation}
            All the face curvature is negative and the cells that are both $g$-cells and $g_\gamma$-cells can never be triangular, hence: 
            \begin{equation}\kappa^f_{\mathcal{D}}(g_{\gamma})\leq\kappa_{\mathcal{D}}^f(g,g_\gamma)+\kappa^t_{\mathcal{D}}(g_{\gamma}),
            \end{equation}
            This further means that:
            \begin{equation}
            \kappa_{\mathcal{D}}^v(g_{\gamma})+\kappa^f_{\mathcal{D}}(g_{\gamma})-\frac{1}{2}\kappa_{\mathcal{D}}^f(g,g_\gamma)\leq\kappa_{\mathcal{D}}^v(g_{\gamma})+\frac{1}{2}\kappa^f_{\mathcal{D}}(g_{\gamma})+\frac{1}{2}\kappa^t_{\mathcal{D}}(g_{\gamma}).
            \label{g gamma approx}
            \end{equation}
            Moreover, by lemma \ref{induced geodesic}, $\kappa_{\mathcal{D}}^v(g_{\gamma})+\frac{1}{2}\kappa_{\mathcal{D}}^f(g_\gamma)+\frac{1}{2}\kappa_{\mathcal{D}}^t(g_\gamma)\leq \frac{2\pi}{3}$. Combining this with the approximation \ref{g gamma approx}, we can further transform inequality \ref{g, g gamma curvature approx} into::

            \begin{equation}
                \kappa_{\mathcal{D}}^v(g)+\kappa^f_{\mathcal{D}}(g)-\frac{1}{2}\kappa^f_{\mathcal{D}}(g,g_{\gamma}) \geq \frac{2\pi}{3}.
                \label{catacomb lowerbound 0}
            \end{equation}

            Now we will have a closer look at the curvature coming from $g$. Assume that there are $m_i$ $n_i$-gons, $n_i\geq6$, that are $g|_{[u',v']}$-cells . Then they must intersect $g|_{[u',v']}$ in a connected sequence of edges, because $g|_{[u',v']}$ is a combinatorial geodesic and the disc diagram is reduced. The edges of $g|_{[u',v']}$ preceding and following this sequence can either belong to another $n_i'$-gon or to a triangle. We differentiate between $n_i$-gons which have zero, one or two $g$-adjacent triangles.
        
            Denote those respectively by $n_i^{(0)}$-gons, $n_i^{(1)}$-gons and $n_i^{(2)}$-gons. Further, denote the number of each of them by $m_i^{(0)}$, $m_i^{(1)}$ and $m_i^{(2)}$, where $m_i^{(0)}+m_i^{(1)}+m_i^{(2)}=m_i.$ Finally, denote the number triangular $g|_{[u',v']}$-cells by $s$. We can see that $s\geq\sum_i\frac{m_i^{(1)}+2m_i^{(2)}}{2}$. As previously mentioned, the curvature of each triangle is $\kappa(t)=-\pi.$ and the curvature of each $n_i^{(j)}$-gon $-(n_i^{(j)}-6)\frac{\pi}{3}$. Additionally, as $\mathcal{D}$ has no cut points, none of the triangles can be shared cells with $g_\gamma$, so we get their full curvature. On the other hand, the $n_i$-gons can be shared, so we might not get their full curvature, but we get at least half of it. Hence, \ref{catacomb lowerbound 0} transforms into the following:
            \begin{equation}
                \kappa_{\mathcal{D}}^v(g)\geq \frac{2\pi}{3}+\frac{1}{2}\sum_{i}\sum_{0\leq j\leq 2}m_i^{(j)}(n_i^{(j)}-6)\frac{\pi}{3}+s\pi.
                \label{catacomb lowerbound}
            \end{equation}
        
            We now analyse the curvature of the vertices on $g|_{(u',v')}$, to give a contradicting upper bound.
        
            By previous analysis, the vertices on the boundary can take curvature $\frac{\pi}{3}-s\frac{2\pi}{3},$ $s\geq0$.
        
            If there are $\frac{n_i^{(0)}}{2}-1$ consecutive vertices of positive curvature on an $n_i^{(0)}$-gon, we can retract over it (see Figure \ref{contracting a cell}), getting a disk of smaller area which satisfies the same conditions and hence allows us to argue by induction. If there are more than $\frac{n_i^{(0)}}{2}-1$ consecutive vertices of positive curvature, a similar move can reduce $g$ to a shorter path, contradicting the fact that it is a combinatorial geodesic. Thus, we can assume that there are at most $\frac{n_i^{(0)}}{2}-2$ consecutive vertices of positive curvature on any $n_i^{(0)}-$gon.
            
            \begin{figure}[!ht]
                \centering
                \resizebox{0.5\textwidth}{!}{%
                \begin{circuitikz}
                \tikzstyle{every node}=[font=\normalsize]
                \draw [short] (3.75,9) -- (6.25,9);
                \draw [short] (6.25,9) -- (7.5,11.5);
                \draw [short] (7.5,11.5) -- (10,11.5);
                \draw [short] (10,11.5) -- (11.25,9);
                \draw [dashed] (6.25,9) -- (7.5,6.5);
                \draw [dashed] (7.5,6.5) -- (10,6.5);
                \draw [dashed] (10,6.5) -- (11.25,9);
                \draw [ color={rgb,255:red,181; green,181; blue,181}, ->, >=Stealth, dashed] (6.5,9.25) -- (6.5,8.75);
                \draw [ color={rgb,255:red,181; green,181; blue,181}, ->, >=Stealth, dashed] (7,10.25) -- (7,7.75);
                \draw [ color={rgb,255:red,181; green,181; blue,181}, ->, >=Stealth, dashed] (7.5,11.25) -- (7.5,6.75);
                \draw [ color={rgb,255:red,181; green,181; blue,181}, ->, >=Stealth, dashed] (8.25,11.5) -- (8.25,6.5);
                \draw [ color={rgb,255:red,181; green,181; blue,181}, ->, >=Stealth, dashed] (9,11.5) -- (9,6.5);
                \draw [ color={rgb,255:red,181; green,181; blue,181}, ->, >=Stealth, dashed] (9.75,11.5) -- (9.75,6.5);
                \draw [ color={rgb,255:red,181; green,181; blue,181}, ->, >=Stealth, dashed] (10.25,10.75) -- (10.25,7.25);
                \draw [ color={rgb,255:red,181; green,181; blue,181}, ->, >=Stealth, dashed] (10.75,9.75) -- (10.75,8.25);
                \draw [short] (11.25,9) -- (13.75,9);
                \node [font=\normalsize] at (7.25,11.75) {$\kappa(v_1)>0$};
                \node [font=\normalsize] at (10,11.75) {$\kappa(v_2)>0$};
                \end{circuitikz}
                }%
                
                \caption{Retracting over a hexagon with two positively curved vertices and no $g$-adjacent triangles.}
                \label{contracting a cell}
            \end{figure}
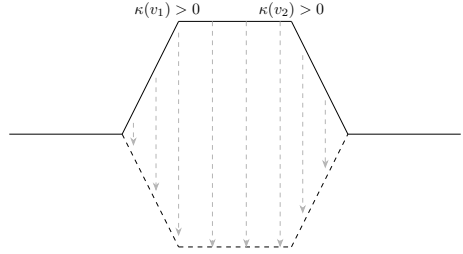
        
            In the case of $n_i^{(1)}$-gons, having $\frac{n_i^{(1)}}{2}-1$ adjacent vertices of curvature $\frac{\pi}{3}$ might not be enough to do a retraction move, because one of the positive vertices might come from a place where the adjacent triangle connects to the cell.  Nevertheless, having $\frac{n_i^{(1)}}{2}$ consecutive vertices of positive curvature is always enough to retract and always produces a shorter path (because the move can be followed by a retraction over a triangle; see the paragraph below), leading to a contradiction. So, we can allow at most $\frac{n_i^{(1)}}{2}-1$ consecutive vertices of positive curvature on any $n_i^{(1)}-$gon.

            Whenever we are in the situation of having an $n$-gon with a $g$-adjacent triangle such that the vertex where they connect is positively curved, we are able to modify the disk diagram so that the $g$-adjacent triangle not only shares an edge with $g$, but also with the $n$-gon. (See Figure \ref{retriangulating}.) Namely, the only other cells incident to that vertex can be triangles, because the one $n$-gon already uses up all the corner angle amount we are allowed for the curvature to stay positive. The sequence of triangles from the $n$-gon to the $g$-adjacent triangle forms an $l$-gon (for some $l$) which is triangulated and can be retriangulated so that the $g$-adjacent triangle also shares an edge with the $n$-gon and none of the properties of the disk diagram are changed (it is still of minimal area and bounded by $g$). 

            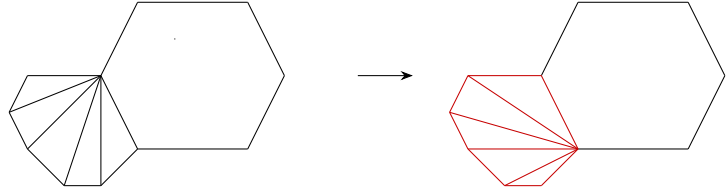
\begin{figure}[!ht]
                
                \centering
                \resizebox{0.8\textwidth}{!}{
                \begin{circuitikz}
                \tikzstyle{every node}=[font=\normalsize]
                \draw [short] (4,12) -- (4,12);
                \draw [short] (3.5,12.5) -- (5,12.5);
                \draw [short] (5,12.5) -- (5.5,11.5);
                \draw [short] (5.5,11.5) -- (5,10.5);
                \draw [short] (5,10.5) -- (3.5,10.5);
                \draw [short] (3.5,10.5) -- (3,11.5);
                \draw [short] (3,11.5) -- (3.5,12.5);
                \draw [short] (1.75,11) -- (3,11.5);
                \draw [short] (3,11.5) -- (2,10.5);
                \draw [short] (2,10.5) -- (1.75,11);
                \draw [short] (2,10.5) -- (2.5,10);
                \draw [short] (2.5,10) -- (3,11.5);
                \draw [short] (3,11.5) -- (3,10);
                \draw [short] (3,10) -- (2.5,10);
                \draw [short] (3,10) -- (3.5,10.5);
                \draw [short] (1.75,11) -- (2,11.5);
                \draw [short] (2,11.5) -- (3,11.5);
                \draw [short] (9,11.5) -- (9.5,12.5);
                \draw [short] (9.5,12.5) -- (11,12.5);
                \draw [short] (11,12.5) -- (11.5,11.5);
                \draw [short] (11.5,11.5) -- (11,10.5);
                \draw [short] (11,10.5) -- (9.5,10.5);
                \draw [ color={rgb,255:red,199; green,0; blue,0}, short] (9,11.5) -- (9.5,10.5);
                \draw [ color={rgb,255:red,199; green,0; blue,0}, short] (9.5,10.5) -- (9,10);
                \draw [ color={rgb,255:red,199; green,0; blue,0}, short] (9.5,10.5) -- (8,11.5);
                \draw [ color={rgb,255:red,199; green,0; blue,0}, short] (8,11.5) -- (9,11.5);
                \draw [ color={rgb,255:red,199; green,0; blue,0}, short] (8,11.5) -- (7.75,11);
                \draw [ color={rgb,255:red,199; green,0; blue,0}, short] (7.75,11) -- (9.5,10.5);
                \draw [ color={rgb,255:red,199; green,0; blue,0}, short] (9.5,10.5) -- (8,10.5);
                \draw [ color={rgb,255:red,199; green,0; blue,0}, short] (8,10.5) -- (7.75,11);
                \draw [ color={rgb,255:red,199; green,0; blue,0}, short] (8,10.5) -- (8.5,10);
                \draw [ color={rgb,255:red,199; green,0; blue,0}, short] (8.5,10) -- (9.5,10.5);
                \draw [ color={rgb,255:red,199; green,0; blue,0}, short] (8.5,10) -- (9,10);
                \draw [->, >=Stealth] (6.5,11.5) -- (7.25,11.5);
                \end{circuitikz}
                }
                \caption{Retriangulating to get a $g$-adjacent triangle which shares an edge with the polygon.}

            \label{retriangulating}
                
            \end{figure}
            
            Finally, for $n_i^{(2)}$-gons, if we account for the $g$-adjacent triangles, we are looking to retract over an $(n_i^{(2)}+2)$-gon. (See Figure \ref{retracting with triangles}.) For this we also need at least $\frac{n_i^{(2)}}{2}$ consecutive vertices of positive curvature. Hence, $n_i^{(2)}$-gons can also admit at most $\frac{n_i^{(2)}}{2}-1$ consecutive vertices of positive curvature.

            \begin{figure}[!ht]
                \centering
                \resizebox{0.35\textwidth}{!}{%
                \begin{circuitikz}
                \tikzstyle{every node}=[font=\normalsize]
                \draw [short] (7.5,12.75) -- (7.5,7.75);
                \draw [short] (7.5,7.75) -- (10,6.5);
                \draw [short] (10,6.5) -- (12.5,7.75);
                \draw [short] (12.5,7.75) -- (12.5,12.75);
                \draw [short] (7.5,7.75) -- (6.25,10.25);
                \draw [short] (12.5,7.75) -- (13.75,10.25);
                \draw [ color={rgb,255:red,199; green,0; blue,0}, line width=0.6pt, short] (6.25,10.25) -- (7.5,12.75);
                \draw [ color={rgb,255:red,199; green,0; blue,0}, line width=0.6pt, short] (7.5,12.75) -- (10,14);
                \draw [ color={rgb,255:red,199; green,0; blue,0}, line width=0.6pt, short] (10,14) -- (12.5,12.75);
                \draw [ color={rgb,255:red,199; green,0; blue,0}, line width=0.6pt, short] (12.5,12.75) -- (13.75,10.25);
                \draw [ color={rgb,255:red,171; green,171; blue,171}, line width=0.6pt, ->, >=Stealth, dashed] (6.75,11.25) -- (6.75,9.5);
                \draw [ color={rgb,255:red,171; green,171; blue,171}, line width=0.6pt, ->, >=Stealth, dashed] (7.25,12) -- (7.25,8.5);
                \draw [ color={rgb,255:red,171; green,171; blue,171}, line width=0.6pt, ->, >=Stealth, dashed] (8.25,13) -- (8.25,7.5);
                \draw [ color={rgb,255:red,171; green,171; blue,171}, line width=0.6pt, ->, >=Stealth, dashed] (9,13.25) -- (9,7.25);
                \draw [ color={rgb,255:red,171; green,171; blue,171}, line width=0.6pt, ->, >=Stealth, dashed] (9.75,13.75) -- (9.75,6.75);
                \draw [ color={rgb,255:red,171; green,171; blue,171}, line width=0.6pt, ->, >=Stealth, dashed] (10.5,13.5) -- (10.5,7);
                \draw [ color={rgb,255:red,171; green,171; blue,171}, line width=0.6pt, ->, >=Stealth, dashed] (11.25,13.25) -- (11.25,7.25);
                \draw [ color={rgb,255:red,171; green,171; blue,171}, line width=0.6pt, ->, >=Stealth, dashed] (12,13) -- (12,7.75);
                \draw [ color={rgb,255:red,171; green,171; blue,171}, line width=0.6pt, ->, >=Stealth, dashed] (12.75,12) -- (12.75,8.5);
                \draw [ color={rgb,255:red,171; green,171; blue,171}, line width=0.6pt, ->, >=Stealth, dashed] (13.25,11) -- (13.25,9.5);
                \node [font=\normalsize, color={rgb,255:red,199; green,0; blue,0}] at (13.5,11.5) {g};
                \node [font=\normalsize, color={rgb,255:red,199; green,0; blue,0}] at (7.25,13.25) {$\kappa$(v_1)>0};
                \node [font=\normalsize, color={rgb,255:red,199; green,0; blue,0}] at (9.75,14.25) {$\kappa$(v_2)>0};
                \node [font=\normalsize, color={rgb,255:red,199; green,0; blue,0}] at (12.5,13.25) {$\kappa$(v_3)>0};
                \end{circuitikz}
                }%
                \caption{Retracting over a polygon with two $g$-adjacent triangles.}

                \label{retracting with triangles}
            \end{figure}
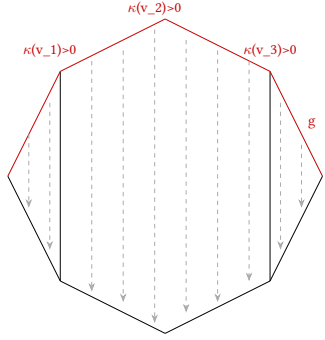

            So far, we produced an upper limit on the number of positively curved vertices on $g|_{(u',v')}$, now we look into what we can say about the negatively curved ones. Each sequence of consecutive vertices of positive curvature on $g|_{(u',v')}$ coming from an $n_i^{(0)}$-gon is between two vertices of negative curvature (unless it is at one of the ends of the segment we're considering). Similarly, each sequence of positive vertices on $g|_{(u',v')}$ coming from an $n_i^{(1)}$-gon, has a negative vertex on one of the sides (unless at an end). Hence, we have at least $\sum_i\frac{2m_i^{(0)}+m_i^{(1)}-2}{2}\frac{\pi}{3}$ negative curvature on $g|_{(u',v')}$.
        
            The observations above imply the following upper bound on the curvature:
            \begin{align}
                \kappa_{\mathcal{D}}^v(g) \leq &\sum_i \left( m_i^{(0)}\left(\frac{n_i^{(0)}}{2}-2\right)\frac{\pi}{3} + m_i^{(1)}\left(\frac{n_i^{(1)}}{2}-1\right)\frac{\pi}{3} + m_i^{(2)} \left(\frac{n_i^{(2)}}{2}-1\right) \frac{\pi}{3}\right) \nonumber\\ 
                &-\sum_i\frac{2m_i^{(0)}+m_i^{(1)}-2}{2}\frac{\pi}{3}.
            \end{align}
            This can further be rewritten as:
            \begin{align}
                \kappa_{\mathcal{D}}^v(g) \leq &\sum_i\left( m_i^{(0)}\left(\frac{n_i^{(0)}}{2}-3\right)\frac{\pi}{3} + m_i^{(1)}\left(\frac{n_i^{(1)}}{2}-\frac{3}{2}\right)\frac{\pi}{3} + m_i^{(2)}\left(\frac{n_i^{(2)}}{2}-1\right) \frac{\pi}{3}\right)\nonumber\\
                &+\frac{\pi}{3} -s\pi+s\pi.
            \end{align}
            Finally, we get:
            \begin{align}
                \kappa_{\mathcal{D}}^v(g) \leq &\sum_i\left( m_i^{(0)}\left(\frac{n_i^{(0)}}{2}-3\right)\frac{\pi}{3} + m_i^{(1)}\left(\frac{n_i^{(1)}}{2}-3\right)\frac{\pi}{3} + m_i^{(2)}\left(\frac{n_i^{(2)}}{2}-4\right)\frac{\pi}{3}\right) \nonumber\\
                &+\frac{\pi}{3}+s\pi.
            \end{align}
            Which can be rewritten as:
            \begin{align}
                \kappa_{\mathcal{D}}^v(g) & \leq \frac{\pi}{3} + \frac{1}{2}\sum_i\sum_{0\leq j\leq2} m_i^{(j)}\left(n_i^{(j)}-6\right)\frac{\pi}{3} +s\pi,
            \end{align}
            contradicting the lower bound \ref{catacomb lowerbound} from before.
        
        \end{proof}

 \section{Cone types of $k$-fold triangle groups}

   In this section, we use Lemma \ref{catacomb} to show the finiteness of cone types for non-positively curved $k$-fold triangle groups. For this purpose, we introduce some additional notions, used to better understand and work with the local and global structures of $X$.
   
    Take $d_e:X^{(2)}\to\mathbb{Z}_{\geq0}$ to be the function on the faces of $X$ induced by the length function on $G$, i.e. distance from the vertex $e$ in $\Cay(G,S)$ (in $X$ this is the distance from a base face). Similarly to $L^*(x)$ defined in Section 3.1, we can define $\mathcal{L}^*(x)$ to be the directed graph on the same set of vertices as $L^*(x)$, where $e=(t_1,t_2)$ is an edge from $t_1\in X^{(2)}$ to $t_2\in X^{(2)}$ if $d_e(t_2)=d_e(t_1)+1$. The map $i$ sends $L^*(x)\subset \Cay(G,S)$ to $i(L^*(x))\subset i(\Cay(G,S))\subset X$.

    Similarly, for $A\subseteq X^{(0)}$, denote by $\mathcal{L}^*(A)=\bigcup_{x\in A}\mathcal{L}^*(x)$.
    \begin{defn}
        We say that a triangle $t\in X^{(2)}$ at $v\in X^{(0)}$ is \textit{minimal} if $d_e(t)\leq d_e(t')$ for every other triangle $t'\in X^{(2)}$ containing $v$.
    \end{defn}
    \begin{rmk}
        As triangles $t\in X^{(2)}$ at $v\in X^{(0)}$ can be identified with vertices of $L^*(v)$ (or equivalently vertices of $\mathcal{L}^*(v)$) we can also talk about \emph{minimal vertices} of $L^*(v)$, respectively $\mathcal{L}^*(v)$.
    \end{rmk}

     We are now ready to state some corollaries of Lemma \ref{catacomb} which elaborate on the interplay between the metric and combinatorial structure of $G$:
    \begin{lemma}
        A $\cat(0)$ geodesic in $X$ between a point $p_e$ in the base face $e$ and a vertex $v\in X^{(0)}$ approaches $v$ from a minimal triangle at $v$. (In other words, the geodesic intersects a minimal triangle at more than just a vertex $v$.) Moreover, if there are more minimal triangles at $v$, they are all adjacent.
        \label {cor1}
    \end{lemma}

    \begin{proof}
        Let $\gamma$ be the $\cat(0)$ geodesic in $X$ between a point $p_e$ in the base face $e$ and a vertex $v\in X^{(0)}$. Assume $\gamma$ approaches $v$ from a triangle $t$ in $X$. 
        
        Suppose that $\gamma$ goes through the interior of $t$. Look at the $\epsilon$-ball around $v$ in $X$, $B(v,\epsilon)$. Take a point $p_{t'}\in B(v,\epsilon)\cap t',$ for some triangle $t'\neq t$ at $v$. Consider the $\cat(0)$ geodesic $\gamma_{t'}$ in $X$ between $p_e$ and $p_{t'}$.

        In a $\cat(0)$ space, geodesics whose ends are close stay uniformly close to each other. Hence, by making $\epsilon$ small enough, we can ensure that $\gamma_{t'}$ crosses the same set of edges as $\gamma$ up to the neigbourhood of $v$, which by Lemma \ref{catacomb} means that $d_e(t)< d_e(t')$. If $\gamma$ passes through some points or edges of $X^{(1)}$ on the way from $p_e$ to $v$, the induced combinatorial path will still have the same number of edges as the combinatorial path induced by $\gamma_{t'}$, because the set of triangles whose interiors they cross is the same conclusive with $t$. 
        
        Suppose now that $\gamma$ approaches $v$ through an edge $[v,w]$ of $t$. There are $k$ triangles at $[v,w]$, $t_1,\dots,t_k$ and their pairwise distances differ by at most one. Now for an arbitrary triangle $t'\in X^{(2)}$, the geodesic $\gamma_{t'}$ (constructed as before) need not necessarily go through the triangle $t$, but needs to either go through the interior of one of the triangles $t_1,\dots,t_k$, denoted by $t_i$, in which case $d(t')> d(t_i)$ or through the edge $[v,w]$, in which case the distance is defined using the combinatorially shortest path. However, this shortest path also needs to go through one of $t_1,\dots,t_k$, as they are the only ones at $v$ connected to $w$ which $\gamma$ also passes through. Either way, we get $d(t')> d(t_i)$, for some $i$. Hence, at least one of the triangles $t_1,\dots,t_k$ needs to be minimal, and it there are more minimal triangles they are all among $t_1,\dots,t_k$. This finishes the proof.
    \end{proof}

    From the lemma above, we see that the distance function is in some sense determined locally. We formalise this by introducing the notion of local distance functions:

    \begin{defn}
       The distance function $d_v:L^*(v)^{(0)}\to\mathbb{Z}_{\geq0}$ is defined as the shortest distance between a vertex $t\in L^*(v)^{(0)}$ and a minimal vertex $t_{v}\in L^*(v)^{(0)}$.
    \end{defn}
    
    \begin{lemma}
        For any triangle $t\in X^{(2)}$ and an adjacent vertex $v\in X^{(0)}$: 
        $$d_e(t)=d_e(t_{v})+d_v(t),$$
        for some minimal triangle $t_{v}$ at $v$.
        \label{cor2}
    \end{lemma}
     This leads to the first theorem of the paper:

\begin{thm}
    \label{cones}
        For every $\tau\in X^{(2)}$, $\mathcal{L}^*(\tau)$ determines the cone type of $\tau$.
    \end{thm}
    By Cannon's theorem (Theorem \ref{Cannon}) we instantly get the following corollary:    \begin{cor}
        The language of all geodesics in $G$ is regular.
    \end{cor}

    \begin{proof}[Proof of Theorem \ref{cones}]
	
	

	From the definition, we can see that the cone of $\tau$ can inductively be constructed by adding the neighbouring faces of greater distance. For a triangle $t\in\cone{\tau}$, $\mathcal{L}^*(t)$ encodes all the local triangles of greater distance. We will show that whenever we add a neighbouring triangle $t'$ of greater distance to $t$, we can determine $\mathcal{L}^*(t')$ from $\mathcal{L}^*(t)$, allowing us to propagate this process and retrieve all the triangles of the cone.
	
	Namely, assume $t=\{x,y,z\}$ is of distance $n$. If $t'=\{y,z,w\}$ is a neighbouring triangle of distance $n+1$ we want to say that we also know $\mathcal{L}^*(t')$, i.e. how $\mathcal{L}^*(w)$ fits together with $\mathcal{L}^*({\{y,z\}})$. We claim that $\mathcal{L}^*(w)$ has to have its minimal vertices in $\mathcal{L}^*(\{y,z\})$. By Lemma \ref{cor2}, we can then get the rest of $\mathcal{L}^*(w)$ by constructing paths to each vertex from the closest minimal one in $L^*(w)$.

    Assume the contrary, i.e. that a minimal triangle $t_w$ at $w$ does not have either $y$ or $z$ as its vertices. Then by Lemma \ref{cor2}, there is a sequence of neighbouring triangles at $w$ from some minimal triangle at $w$ to $t'$ increasing in distance. Furthermore, Lemma \ref{cor1} tells us that all minimal triangles at a vertex are adjacent, so we can see this sequence as starting from $t_w$, where the first step is not necessarily distance increasing, but all the others are. Without loss of generality, we can assume that the second-to-last member of that sequence is a triangle $t''$ sharing vertices $w$ and $z$ with $t'$ (see Figure \ref{min triangle}). The distance function at $t''$ is $n$.
        
    \begin{figure}[!ht]
        \centering
        \resizebox{0.4\textwidth}{!}{%
        \begin{circuitikz}
        \tikzstyle{every node}=[font=\normalsize]
        \draw [short] (7,11) -- (8,9);
        \draw [short] (8,9) -- (6,9);
        \draw [short] (6,9) -- (7,11);
        \draw [short] (7,11) -- (5,11);
        \draw [short] (5,11) -- (6,9);
        \draw [short] (5,11) -- (3,10);
        \draw [short] (5,11) -- (3,12);
        \draw [short] (3,12) -- (3,10);
        \draw [dashed] (5,11) -- (4,9);
        \draw [dashed] (4,9) -- (6,9);
        \node [font=\normalsize] at (8.25,8.75) {x};
        \node [font=\normalsize] at (7,11.25) {y};
        \node [font=\normalsize] at (6,8.75) {z};
        \node [font=\normalsize] at (5,11.25) {w};
        \node [font=\normalsize] at (7,9.75) {t};
        \node [font=\normalsize] at (6,10.25) {t'};
        \node [font=\normalsize] at (3.5,11) {$t_w$};
        \node [font=\normalsize] at (5,9.75) {t''};
        \end{circuitikz}
        }%
        \caption{Triangles $t$, $t'$, $t''$ and $t_w$}
        \label{min triangle}
    \end{figure}
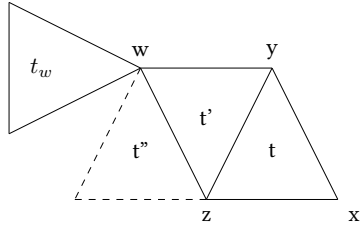

    Denote the third vertex of $t''$ by $v$. Now look at the $\cat(0)$ geodesics $\gamma_v$ and $\gamma_z$ from $p_e$ to the vertices $v$ and $z$ respectively. Lemma \ref{cor1} tells us these geodesics must approach their respective end vertices through minimal triangles. We further know that smaller distance triangles surround the triangle $t'$ from both sides. Take the shortest path $p$ in $L^*(z)$ from a minimal vertex of $L^*(z)$ to $t$ and the shortest path $p''$ in $L^*(z)$ from a minimal vertex of $L^*(z)$ to $t''$. Because $t$ and $t''$ are of the same distance, paths $p$ and $p''$ are of the same length. Since minimal points of $L^*(z)$ are at most distance one apart, we can use these paths to complete a cycle $c$ in $L^*(z)$ containing the points $t'',t',t$. As $c$ is a cycle in $L^*(z)$, its girth is at least 6, so the length of each of the paths is at least 2. Hence, there is at least one triangle between $t''$ and any minimal triangle at $z$. This implies that the angle between the edge $\{v,z\}$ and $\gamma_z$ is larger than $\frac{\pi}{3}$. 
    Similarly, triangle $t''$ has a smaller distance triangle on one side and a smaller or equal distance triangle on the other side. If both are smaller, analogous to the argument above, we have a path of distance at least 3 from the minimal triangle at $v$ to $t''$, so the angle between the edge $\gamma_v$ and $\{v,z\}$ is larger than $\frac{2\pi}{3}$.
    If there is a triangle of equal distance on one side of $t''$ at $v$, then that triangle and $t''$ have to be the minimal triangles at $w$, so the other triangles containing $v,w$ cannot be of smaller distance. Hence, we can use the same approach of looking at the paths to the two triangles and the cycle $c'$ induced by them, which will not be trivial because the paths cannot enter the triangles from the same edge. In this case, the path from a minimal triangle at $v$ to $t''$ is at least 2, but it can only be two if $L^*(v)$ contains two minimal vertices. This would further mean that $\gamma_v$ approaches $v$ at the intersection of the two minimal triangles, so the angle between the edge $\gamma_v$ and $\{v,z\}$ would be equal to $\frac{2\pi}{3}$. 
    To sum up, the angle between the edge $\gamma_v$ and $\{v,z\}$ is larger or equal to $\frac{2\pi}{3}$ and the angle between the edge $\{v,z\}$ and $\gamma_z$ is larger than $\frac{\pi}{3}$, so the sum of the angles of the triangle $\{z,v,p_e\}$ (see Figure \ref{contradictory triangle}) is greater than $\pi$, which is a contradiction.

    \begin{figure}[!ht]
        \centering
        \resizebox{0.5\textwidth}{!}{%
        \begin{circuitikz}
        \tikzstyle{every node}=[font=\normalsize]
        \draw [short] (7,11) -- (8,9);
        \draw [short] (8,9) -- (6,9);
        \draw [short] (6,9) -- (7,11);
        \draw [short] (7,11) -- (5,11);
        \draw [short] (5,11) -- (6,9);
        \node [font=\normalsize] at (8.25,8.75) {x};
        \node [font=\normalsize] at (7,11.25) {y};
        \node [font=\normalsize] at (6,8.75) {z};
        \node [font=\normalsize] at (5,11.25) {w};
        \node [font=\normalsize] at (7,9.75) {t};
        \node [font=\normalsize] at (6,10.25) {t'};
        \node [font=\normalsize] at (5,9.75) {t''};
        \draw [short] (6,9) -- (4,9);
        \draw [short] (4,9) -- (5,11);
        \node [font=\normalsize] at (4,8.75) {v};
        \draw [short] (4,9) -- (0.5,6);
        \draw [short] (0.5,6) -- (6,9);
        \node [font=\normalsize] at (0.5,5.75) {$p_e$};
        \end{circuitikz}
        }%
        \caption{The triangle $\{p_e,v,z\}$}
        \label{contradictory triangle}
    \end{figure}
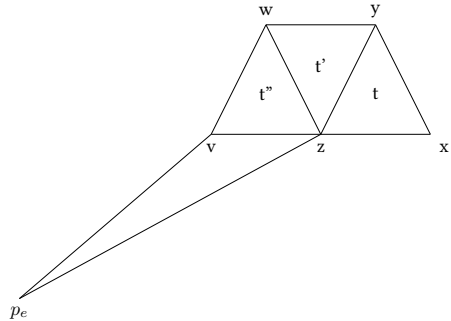
        
    To sum up, we proved that $\mathcal{L}^*(t)$ determines $\mathcal{L}^*(t')$, which inductively gives us the cone type of $\tau$.
\end{proof}

\section{The language of lexicographically first geodesics}

In this section, we show that the language of lexicographically first geodesics in $k$-fold triangle groups satisfies the fellow traveller property and hence, provides an automatic structure for $G$.

\begin{defn}
    Let $\mathcal{D}$ be a disk diagram over $C$, and $g$ a combinatorial geodesic such that $g\subset\partial\mathcal{D}$. We say that an $n$-gonal $g$-cell $c\in\mathcal{D}$, $n\geq6$, is \textit{$g$-maximally positively curved} if it contains the maximal possible number of vertices of positive curvature. In the case of $n$-gonal $g$-cells with at most one $g$-adjacent triangle this number is $\frac{n}{2}-1$, and in the case of $n$-gonal $g$-cells with two $g$-adjacent triangles it is $\frac{n}{2}$. We say that $c$ is \textit{almost $g$-maximally positively curved} if it contains one fewer positively curved vertices than possible.
\end{defn}

    In Lemma \ref{catacomb}, the $g$-maximally positively curved polygons with no $g$-adjacent triangles were the ones over which we were able to retract to reduce the area of the disc diagram. In this section, we will not be able to do that. Still, we will prove that if $g$ is a combinatorial geodesic, then there has to be additional negative curvature between such polygons.

\begin{lemma}
    Let $\mathcal{D}$ be a disk diagram over $C$, and $g$ a combinatorial geodesic such that $g\subset\partial\mathcal{D}$ and $g$ contains no cut points of $\mathcal{D}$. Further, let all the non-triangular $g$-cells be $g$-maximally positively curved with the exception of the ones with no $g$-adjacent triangles, which are also allowed to be almost $g$-maximally positively curved. Then for every two $g$-maximally positively curved cells with no $g$-adjacent triangles, there is a cell $c\in\mathcal{D}$ between them, which intersects $g$ but is not a $g$-cell.
    \label{maximally positive}
    
\end{lemma}

\begin{proof}
    
    Let $c_1$ and $c_2$ be two $g$-maximally positively curved cells with no $g$-adjacent triangles with no other such cells between them.
    We will look at a disc diagram $\mathcal{D}'$  consisting of all cells intersecting $g$ between $c_1$ and $c_2$. Assume that $\mathcal{D}'$ consists only of $g$-cells.

    Because all the non-triangular $g$-cells are (almost) $g$-maximally positively curved, all of them, apart from $c_1$ and $c_2$, have at least as many vertices on $g$ as they do on $\partial \mathcal{D}'-g$. (See Figure \ref{ladder argument}) But $c_1$ and $c_2$ have strictly more vertices on $g$ than on $\partial \mathcal{D}'-g$. This contradicts $g$ being a combinatorial geodesic.
    \begin{figure}[!ht]
        \centering
        \resizebox{0.6\textwidth}{!}{%
        \begin{circuitikz}
        \tikzstyle{every node}=[font=\LARGE]
        \draw [ color={rgb,255:red,180; green,0; blue,0}, short] (6.25,6.5) -- (6.25,5.25);
        \draw [ color={rgb,255:red,0; green,180; blue,180}, short] (6.25,6.5) -- (7.5,7.25);
        \draw [ color={rgb,255:red,0; green,180; blue,180}, short] (7.5,7.25) -- (8.75,6.5);
        \draw [short] (8.75,6.5) -- (8.75,5.25);
        \draw [ color={rgb,255:red,180; green,0; blue,0}, short] (8.75,5.25) -- (7.5,4.5);
        \draw [ color={rgb,255:red,180; green,0; blue,0}, short] (7.5,4.5) -- (6.25,5.25);
        \draw [ color={rgb,255:red,0; green,180; blue,180}, short] (8.75,6.5) -- (10,7.25);
        \draw [ color={rgb,255:red,0; green,180; blue,180}, short] (10,7.25) -- (11.25,6.5);
        \draw [short] (11.25,6.5) -- (11.25,5.25);
        \draw [ color={rgb,255:red,180; green,0; blue,0}, short] (11.25,5.25) -- (10,4.5);
        \draw [ color={rgb,255:red,180; green,0; blue,0}, short] (10,4.5) -- (8.75,5.25);
        \draw [short] (11.25,6.5) -- (12.5,5.75);
        \draw [ color={rgb,255:red,180; green,0; blue,0}, short] (12.5,5.75) -- (11.25,5.25);
        \draw [ color={rgb,255:red,0; green,180; blue,180}, short] (11.25,6.5) -- (11.25,7.75);
        \draw [ color={rgb,255:red,0; green,180; blue,180}, short] (11.25,7.75) -- (12.5,8.5);
        \draw [ color={rgb,255:red,0; green,0; blue,0}, short] (12.5,8.5) -- (13.75,7.75);
        \draw [ color={rgb,255:red,180; green,0; blue,0}, short] (13.75,7.75) -- (13.75,6.5);
        \draw [ color={rgb,255:red,180; green,0; blue,0}, short] (13.75,6.5) -- (12.5,5.75);
        \draw [ color={rgb,255:red,0; green,180; blue,180}, short] (12.5,8.5) -- (12.5,9.75);
        \draw [ color={rgb,255:red,0; green,180; blue,180}, short] (12.5,9.75)-- (13.75,10.5);
        \draw [ color={rgb,255:red,180; green,0; blue,0}, short] (13.75,10.5) -- (15.00,9.75);
        \draw [ color={rgb,255:red,180; green,0; blue,0}, short] (15.00,9.75) -- (15.00,8.5);
        \draw [ color={rgb,255:red,180; green,0; blue,0}, short] (15.00,8.5) -- (13.75,7.75);
        \node [font=\normalsize, color={rgb,255:red,180; green,0; blue,0}] at (11.9,5) {$g$};
        \node [font=\normalsize, color={rgb,255:red,13; green,13; blue,13}] at (7.5,6) {$c_1$};
        \node [font=\normalsize, color={rgb,255:red,13; green,13; blue,13}] at (13.75,9.25) {$c_2$};
        \node [font=\normalsize, color={rgb,255:red,0; green,180; blue,180}] at (10.25,7.75) {$\partial \mathcal{D}'-g$};
        \end{circuitikz}
        }%
        
        \caption{The disk diagram $\mathcal{D}'$ consists only of $g$-maximally positively curved $g$-cells and their $g$-adjacent triangles.}
        \label{ladder argument}
    \end{figure}
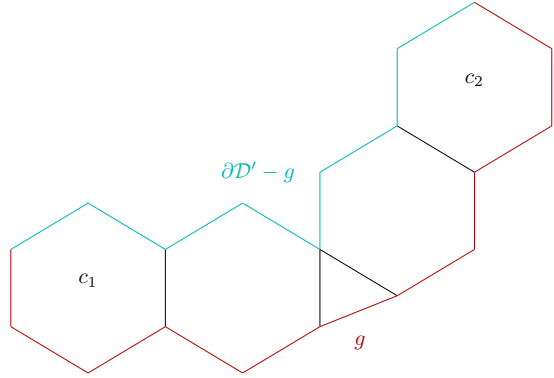
\end{proof}

\begin{rmk}
    When $g$ is a combinatorial geodesic contained in the boundary of a disk diagram $\mathcal{D}$, the curvature of $\mathcal{D}$ that can be see as coming from $g$ is:
\begin{equation}
                 \kappa_{\mathcal{D}}(g):=\kappa^v_{\mathcal{D}}(g)+\kappa^f_{\mathcal{D}}(g) + \sum_{\substack{f\in \mathcal{D}^{(2)}\\ f\cap g \neq \emptyset\\ f \text{ is not a $g$-cell}}} \kappa(f).
                 \label{lemma demo1}
            \end{equation}
As all the faces are non-posiitively curved, we usually only want to consider $g$-cells, i.e. we use the following bound:

\begin{equation}
                 \kappa_{\mathcal{D}}(g)\leq\sum_{v\in \interior(g)} (\pi-\sum_{\substack{f_v \text{ is a }g\text{-cell containing } v\\f_v \text{ not tiangular}}}\frac{2\pi}{3})+\kappa^f_{\mathcal{D}}(g).
                 \label{lemma demo2}
\end{equation}

            The claim of Lemma \ref{maximally positive} is that when all the $g$-cells are $g$-maximally positively curved and there are at least two without $g$-adjacent triangles, then we lost something in passing from the equation \ref{lemma demo1} to the inequality \ref{lemma demo2}. Thus, we can lower the upper bound \ref{lemma demo2} by detecting more negative curvature, which comes from the cell $c$ as defined in Lemma \ref{maximally positive}.

            Namely, if $c$ is a triangle, then by accounting for its curvature, we can make the upper bound \ref{lemma demo2} smaller for $\pi$.

            Otherwise, if $c$ is not a triangle, its face curvature might not change the upper bound \ref{lemma demo2}, but the vertex in which $c$ intersects $g$ will have at least $\frac{2\pi}{3}$ smaller curvature than if $c$ did not exist.

            To conclude, in the case of the lemma, the inequality \ref{lemma demo2} can be improved to:

            \begin{equation}
                \kappa_{\mathcal{D}}(g)\leq\sum_{v\in \interior(g)} (\pi-\sum_{\substack{f_v \text{ is a }g\text{-cell containing } v\\f_v \text{ not tiangular}}}\frac{2\pi}{3})+\kappa^f_{\mathcal{D}}(g)- \frac{2\pi}{3}.
                \label{gmpc bound improved}
            \end{equation}
            
\end{rmk}

\begin{lemma}
    A combinatorial geodesic in $\Cay(G,S)$ from $e$ cannot enter $L^*(x)$, for any $x$, at a (combinatorial) distance greater than one from any minimal vertex of $L^*(x)$.
    \label{comb_geod_enters}
\end{lemma}

\begin{proof}
    Let $g$ be a combinatorial geodesic from $e$ to $t$. Assume that $t\in L^*(x)$, for some $x$, is of distance at least 2 from a minimal triangle $t_x$ of $L^*(x)$, and no other vertices of $g$ are in $L^*(x)$. Then the vertex $t'$, preceding $t$ on $g$ is of smaller distance than $t$. So, $t$ is surrounded by smaller distance vertices from two sides. Hence, by the same argument as in the proof of Theorem \ref{cones}, there is a vertex $y$ of $X$ such that $t,t'\in L^*(y)$ and a $\cat(0)$ geodesic $\gamma_y$ from $e$ to $y$ such that the angle between $\gamma_y$ and $[x,y]$ is greater than $\frac{2\pi}{3}$. Furthermore, the fact that $d_x(t_x,t)\geq2$ means that for the geodesic $\gamma_x$ from $e$ to $x$, the angle between $\gamma_x$ and $[x,y]$ is greater or equal to $\frac{\pi}{3}$. But that means that the triangle $\{e,x,y\}$ has a sum of angles greater than $\pi$, which is impossible.
\end{proof}

\begin{thm}
    The lexicographically first geodesics in $\Cay(G,S)$ satisfy the fellow traveller property.
\end{thm}

\begin{proof}
    Let $g$ be the lexicographically first geodesic from $u\in\Cay(G,S)^{(0)}$ to $w\in\Cay(G,S)^{(0)}$ and $g'$ the lexicographically first geodesic from $u$ to  $v=ws\in\Cay(G,S)$, for some $s\in S$. We will show that such two geodesics are always at most distance $\delta=\max_{x\in X^{(0)}}\diam L^*(x)$, which is finite because $L^*(x)$ can take at most three different forms (one for each vertex group of $G$).

    This uniform bound (which gives what is usually referred to as the \textit{asynchronous} fellow traveller property) in the case of a geodesic language implies the fellow traveller property.

    Denote the vertices on the geodesic $g$ by $u_0=u,u_1,\dots,u_n=w$ and the edges by $s_1,\dots,s_n$. Similarly, denote the vertices of $g'$ by $u'_0=u,\dots,u'_m=v$ and the edges by $s'_1,\dots,s'_m$. (See Figure \ref{lexgeod}.)

    Assume now, without the loss of generality, that $s_1<s_1'$. Observe that this implies that $v\notin \overline{\cone{u_1}}$, as otherwise $g'$ would not be a lexicographically first geodesic. Furthermore, no vertices of $g'$ can belong to $\overline{\cone{u_1}}$, as otherwise we could get a combinatorial geodesic with the same endpoints and of the same length starting with a smaller letter, contradicting that $g'$ is lexicographically first. On the other hand, all the vertices of $g$ need to be in $\overline{\cone{u_1}}$, so $u_i\neq u_j'$ for every $i$ and $j$.
    
    We will use a disk diagram argument to show that $g$ and $g'$ stay close together.
    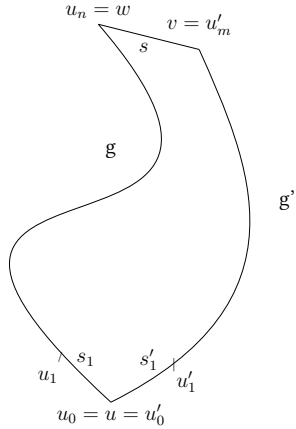
\begin{figure}[!ht]
        \centering
        \resizebox{0.6\textwidth}{!}{%
        \begin{circuitikz}
        \tikzstyle{every node}=[font=\large]
        \draw [short] (6.25,5.25) .. controls (0,11) and (10.75,7.25) .. (6,12.75);
        \node [font=\large] at (6.25,10.25) {g};
        \node [font=\large] at (9.75,9.25) {g'};
        \node [font=\large] at (6.25,5) {$u_0=u=u_0'$};
        \draw [short] (6.25,5.25) .. controls (10.25,7.25) and (9,10) .. (8,12.25);
        \draw [short] (6,12.75) -- (8,12.25);
        \node [font=\large] at (8,12.75) {$v=u_m'$};
        \node [font=\footnotesize] at (5.25,6.15) { / };
        \node [font=\footnotesize] at (7.5,6) { | };
        \node [font=\large] at (6,13) {$u_n=w$};
        \node [font=\large] at (6.9,12.3) {$s$};
        \node [font=\large] at (5.75,6.1) {$s_1$};
        \node [font=\large] at (7,6.1) {$s_1'$};
        \node [font=\large] at (5,5.75) {$u_1$};
        \node [font=\large] at (7.75,5.65) {$u_1'$};
        
        \end{circuitikz}
        }%
        \caption{lexicographically first geodesics $g$ and $g'$}
        \label{lexgeod}
    \end{figure}

     Let $x\in X^{(0)}$ be the farthest point from $u$ such that $g$ and $g'$ enter it at a distance greater than 1. Note that by Lemma \ref{comb_geod_enters}, this distance needs to be exactly 2. If such $x$ doesn't exist, then $g$ and $g'$ enter and consequently exit every $L^*(x)$ at distance one, so they are never at distance larger than $\delta$ from each other and the proof is done.
     
     Now, let $w'$ and $v'$ be the vertices of $g$ and $g'$, respectively, where $g$ and $g'$ enter $L^*(x)$. Consider a disk diagram $\mathcal{D}$ in the cell complex $C$ of minimal area bounded by $g$, $g'$, and the geodesic path $l$ in $L^*(x)$ connecting $w'$ and $v'$.

     As noted above, the disk diagram $\mathcal{D}$, has no cut-vertices, but we don't want to have any cut-triangles either, i.e. triangular cells in $\mathcal{D}$ that have two vertices on $g$ and one vertex on $g'$ or vice-versa. If such a triangle exists, we can split our disk diagram into multiple diagrams that behave in the same way.

     Namely, if the cut triangle $t$ has two vertices $u_k$ and $u_{k+1}$ on $g$ and one vertex $u_j'$ on $g'$, because of the cone condition, the one on $g'$ has to be of distance $k$ from $e$, so it has to be $u_k'$. Furthermore, $\overline{\cone{u_{k+1}}}\subseteq \overline{\cone{u_1}}$. Hence, instead of looking at a disk diagram $\mathcal{D}$, we can look at two disk diagrams $\mathcal{D}_1$ and $\mathcal{D}_2$ that behave in the same way. The first one, $\mathcal{D}_1$, starts at $u$ and is bounded by $g$ and $g'$ up to the farthest point from $u$ \emph{before the triangle $t$} such that $g$ and $g'$ enter it at a distance greater than 1. The second one, $\mathcal{D}_2$, starts at $u_k'$, which is the new base point, and is bounded by a section of $g'$ from $u_k'$ to $u_m'$ and a section of $g$ from $u_{k+1}$ to $u_n$, where our new combinatorial geodesic $g$ will start at $u_k'$ and then follow along $g$. This is still a combinatorial geodesic and the cone condition at $u_1$ from before translates to the cone condition at $u_{k+1}.$

     In the other case, when the cut triangle $t$ has one vertex $u_k$ on $g$ and two vertices $u_j'$ and $u_{j+1}'$ on $g'$, because of the cone condition, the one on $g$ has to be of distance $j+1$ from $e$, so it has to be $u_{j+1}$. Furthermore, $\overline{\cone{u_{j+1}}}\subset\overline{\cone{u_1}}$. We can again split $\mathcal{D}$ into two disk diagrams $\mathcal{D}_1$ and $\mathcal{D}_2$ that behave in the same way. The first one, $\mathcal{D}_1$, starts at $u$ and is bounded by $g$ and $g'$ up to the farthest point from $u$ \emph{before the triangle $t$} such that $g$ and $g'$ enter it at a distance greater than 1. The second one, $\mathcal{D}_2$, starts at $u_j'$, which is the new base point, and is bounded by a section of $g'$ from $u_j'$ to $u_m'$ and a section of $g$ from $u_{j+1}$ to $u_n$, where our new combinatorial geodesic $g$ will start at $u_j'$ and then follow along $g$. This is still a combinatorial geodesic and the cone condition at $u_1$ from before translates to the cone condition at $u_{j+1}.$

     Because the disks $\mathcal{D}_i$ constructed as above have the same properties as $\mathcal{D}$, in the rest of the proof we can assume, we are working with $\mathcal{D}$ without any cut triangles.

    We now estimate the curvatures. By the combinatorial Gauss--Bonnet theorem for $\mathcal{D}$:
    \begin{align}
        2\pi\chi(\mathcal{D})=&\sum_{v\in \mathcal{D}^{(0)}} \kappa(v) + \sum_{f\in \mathcal{D}^{(2)}} \kappa(f)  \nonumber \\ 
        =&\kappa_{\mathcal{D}}^v(g)+\kappa_{\mathcal{D}}^v(l)+\kappa_{\mathcal{D}}^v(g') +\kappa(u)+\kappa(v')+\kappa(w')+\sum_{f\in \mathcal{D}^{(2)}} \kappa(f),
        \label{first estimate}
    \end{align}
    where we can split the vertices of $\mathcal{D}$ in the way above because $g$ and $g'$ don't share any vertices.
    
    The interior angle at $u$ can either be $\frac{2\pi}{3}$, in which case $\kappa(u)\leq\frac{\pi}{3}$, or it can be zero if both geodesics pass through the same first edge in $X$, in which case $\kappa(u)=\pi$. We want to avoid the case of angle zero. This is done in the following way: If we are in the situation of angle zero, instead of looking at the disc diagram $\mathcal{D}$, we look at a modified disc diagram $\mathcal{\tilde{D}}$. The modified disc is bounded by $l,g$ and $\tilde{g}'$, where $\tilde{g}'$, starts from $u_1$ instead of $u$. The new path $\tilde{g}'$ is still a combinatorial geodesic, because having a shorter path to some point $p$ of $g'$ from $u_1$ would imply that that point is in $\overline{\cone{u_1}}$. But now we must have an interior angle of at least $\frac{2\pi}{3}$ at the starting vertex. Hence, without loss of generality, we can assume that $\kappa(u)\leq\frac{\pi}{3}$. Furthermore, the interior angles of $\mathcal{D}$ at $v'$ and $w'$ are $\frac{2\pi}{3}$, so we know that $\kappa(v'),\kappa(w')\leq\frac{\pi}{3}$, but we won't use these estimates yet because of potential overlap with later estimates. We also know that $\chi(\mathcal{D})=1$, so we can rewrite \ref{first estimate} as:
    \begin{equation}
        2\pi\leq \kappa_{\mathcal{D}}^v(g)+\kappa_{\mathcal{D}}^v(l)+\kappa_{\mathcal{D}}^v(g')+\sum_{f\in \mathcal{D}^{(2)}} \kappa(f)+\kappa(w')+ \kappa(v')+\frac{\pi}{3}
        \label{lex_curv_est}.
    \end{equation}

    The faces of $\mathcal{D}$ can be $g$-cells, $g'$-cells, both or neither, so 
    \begin{equation}
        \sum_{f\in \mathcal{D}^{(2)}} \kappa(f)= \kappa_{\mathcal{D}}^f(g)+\kappa_{\mathcal{D}}^f(g')-\kappa_{\mathcal{D}}^f(g,g')+\sum_{\substack{f\in \mathcal{D}^{(2)}\\ f \text{ is not a $g$-cell nor a $g'$-cell} }}\kappa(f).
    \end{equation}

    As all the faces are non-positively curved, the estimate doesn't have to include all of them, so we will only keep the $g$-cells, the $g'$-cells and the triangles coming from the lemma \ref{maximally positive} applied to $g$ and $g'$. The estimate \ref{lex_curv_est} becomes the following:
    \begin{align}
        \pi\leq& \kappa_{\mathcal{D}}^v(g)+\kappa_{\mathcal{D}}^f(g)+ \sum_{\substack{c\in \mathcal{D}^{(2)}\\\text{triangle coming from }\\ \text{Lemma \ref{maximally positive} applied to $g$}}} \kappa(c)-\frac{1}{2}\kappa_{\mathcal{D}}^f(g,g')+\kappa_{\mathcal{D}}^v(g')+\kappa_{\mathcal{D}}^f(g)\nonumber\\ 
        &+ \sum_{\substack{c\in \mathcal{D}^{(2)}\\\text{triangle coming from }\\ \text{Lemma \ref{maximally positive} applied to $g'$}}} \kappa(c)-\frac{1}{2}\kappa_{\mathcal{D}}^f(g,g')+ \kappa_{\mathcal{D}}^v(l)
        +\kappa(w')+\kappa(v')
        \label{lex_curv_est2}.
    \end{align}

    The reason why the cells $c$ coming from Lemma \ref{maximally positive} can be split into the ones coming from $g$ and the ones coming from $g'$ cleanly, i.e. without worrying about double counting is the following: Assume that there is a triangle $c$ as in Lemma \ref{maximally positive} which is the source of the extra negative curvature for both $g$ and $g'$. Then such a triangle has exactly one vertex $u_k$ on $g$ and exactly one vertex $u_j'$ on $g'$. As we are in the context of Lemma \ref{maximally positive}, we can assume that everything around those vertices is as positively curved as possible including them, so they are both contained in exactly two non-triangular cells of the disk diagram $\mathcal{D}$. Furthermore the cells at $u_k$ and $u_j'$ containing the edge $\{u_k,u_j'\}$ of the triangle $c$ have to be the same cell. So, at total there are three cells around the triangle $c$, two of which have to belong to the same $L(x)$, where $x$ is one of the two endpoints of an edge in $X$ around which $c$ lives. Without loss of generality assume that those two cells are $g$-cells. This further implies that both edges of $g$ containing $u_k$ have to pass through the same edge in $X$, but then $g$ is not a combinatorial geodesic, which is a contradiction.

    For estimating the curvature of the geodesic segments $g$ and $g'$, we use the same vertex curvature analysis as in Lemma \ref{catacomb}. We start with $g'$. As before, denote by ${m'}_i^{(j)}$-gons the number of ${n'}_i^{(j)}$-gons, ${n'}_i^{(j)}\geq 6$, that are $g'$-cells with $j\in\{0,1,2\}$ $g'$-adjacent triangles. Further, denote by $s'$ the number of $g'$-adjacent triangles. Unlike in Lemma \ref{catacomb} we are not looking to retract the disk diagram, so the bounds on the maximal positive curvature of $g'$-cells will be different.
    
    As $g'$ is a combinatorial geodesic, the maximal number of positive vertices on an ${n'}_i^{(0)}$-gon is $\frac{{n'}_i^{(0)}}{2}-1$. For ${n'}_i^{(1)}$-gons it is $\frac{{n'}_i^{(1)}}{2}-1$, and for ${n'}_i^{(2)}$-gons it is $\frac{{n'}_i^{(2)}}{2}$.
    
    Also, as before, for every ${n'}_i^{(0)}$-gon, we get two negatively curved vertices on $g'$ and for every ${n'}_i^{(1)}$-gon one, so, by excluding the end vertices, we get that there are is at least $\sum_i\frac{2{m'}_i^{(0)}+{m'}_i^{(1)}-2}{2}\frac{\pi}{3}$ negative curvature on $g'$. 
    
    Finally, any ${n'}^{(0)}$-gon is either not $g'$-maximally positively curved, which means that the total curvature over $g'$ is $\pi/3$ smaller than it could be, or it is $g'$-maximally positively curved and, by Lemma \ref{maximally positive}, produces $-\frac{\pi}{3}$ extra curvature with another $g'$-maximally positively curved ${n'}^{(0)}$-gon, if such exists. This extra negative curvature either comes from a non-$g$-cell $c$ between the ${n'}^{(0)}$-gons or from one of the $g$-cells between them not being maximally curved in this context. It is encoded either through vertex curvature, which means it is in the sum $\kappa_{\mathcal{D}}^v(g')$, or it is encoded through the face curvature of a triangular cell that is not a $g$-cell. There are no cut points on $g'$ because of how $\mathcal{D}$ was defined. In total, there is at least $({m'}_i^{(0)}-1)\frac{\pi}{3}$ of extra negative curvature for every $i$.

    As for the curvature of $g'$-adjacent faces that was not accounted for above (as a cell $c$ from Lemma \ref{maximally positive}), we have $g'$-cells contributing $-\sum_i\sum_{0\leq j\leq2} {m'}_i^{(j)}\left({n'}_i^{(j)}-6\right)\frac{\pi}{3}$ to the curvature and their $g'$-adjacent triangles contributing $-s'\pi$.

    Lastly, $\kappa_{\mathcal{D}}^f(g,g')$ can be at most $ \sum_i\sum_{0\leq j\leq2} {m'}_i^{(j)}\left({n'}_i^{(j)}-6\right)\frac{\pi}{3}$ if all of the non-triangular $g'$-cells are also $g$-cells.
    
    Adding up all those estimates, we get the following upper bound:
     \begin{align}
    \kappa_{\mathcal{D}}^v(g')+&\kappa_{\mathcal{D}}^f(g')+ \sum_{\substack{c\in \mathcal{D}^{(2)}\\\text{triangle coming from }\\ \text{Lemma \ref{maximally positive} applied to $g'$}}} \kappa(c) -\frac{1}{2}\kappa_{\mathcal{D}}^f(g,g')\leq \nonumber\\ 
    &\sum_i \left( {m'}_i^{(0)}\left(\frac{{n'}_i^{(0)}}{2}-1\right)\frac{\pi}{3} + {m'}_i^{(1)}\left(\frac{{n'}_i^{(1)}}{2}-1\right)\frac{\pi}{3} + {m'}_i^{(2)} \left(\frac{{n'}_i^{(2)}}{2}\right) \frac{\pi}{3}\right) \nonumber \\
                & -\sum_i\frac{2{m'}_i^{(0)}+{m'}_i^{(1)}-2}{2}\frac{\pi}{3}-(\sum_i{m'}_i^{(0)}-1)\frac{\pi}{3} \nonumber  \\&- \sum_i\sum_{0\leq j\leq2} {m'}_i^{(j)}\left({n'}_i^{(j)}-6\right)\frac{\pi}{3} -s'\pi +\frac{1}{2}\sum_i\sum_{0\leq j\leq2} {m'}_i^{(j)}\left({n'}_i^{(j)}-6\right)\frac{\pi}{3}.
    \end{align}

    This further transforms into:
    \begin{align}
    \kappa_{\mathcal{D}}^v(g')+&\kappa_{\mathcal{D}}^f(g')+ \sum_{\substack{c\in \mathcal{D}^{(2)}\\\text{triangle coming from }\\ \text{Lemma \ref{maximally positive} applied to $g'$}}} \kappa(c) -\frac{1}{2}\kappa_{\mathcal{D}}^f(g,g') \leq \nonumber\\
                 &\sum_i\left( {m'}_i^{(0)}\left(\frac{{n'}_i^{(0)}}{2}-3\right)\frac{\pi}{3} + {m'}_i^{(1)}\left(\frac{{n'}_i^{(1)}}{2}-\frac{3}{2}\right)\frac{\pi}{3} + {m'}_i^{(2)}\left(\frac{{n'}_i^{(2)}}{2}\right) \frac{\pi}{3}\right) \nonumber  \\
                &+\frac{2\pi}{3}- \frac{1}{2}\sum_i\sum_{0\leq j\leq2} {m'}_i^{(j)}\left({n'}_i^{(j)}-6\right)\frac{\pi}{3} -s'\pi.
    \end{align}

    Moreover, using that $s'\geq\sum_i\frac{{m'}_i^{(1)}+2{m'}_i^{(2)}}{2}$, we get:
    \begin{align}
    \kappa_{\mathcal{D}}^v(g')+&\kappa_{\mathcal{D}}^f(g')+ \sum_{\substack{c\in \mathcal{D}^{(2)}\\\text{triangle coming from }\\ \text{Lemma \ref{maximally positive} applied to $g'$}}} \kappa(c) -\frac{1}{2}\kappa_{\mathcal{D}}^f(g,g') \leq \nonumber\\
                & \sum_i\left( {m'}_i^{(0)}\left(\frac{{n'}_i^{(0)}}{2}-3\right)\frac{\pi}{3} + {m'}_i^{(1)}\left(\frac{{n'}_i^{(1)}}{2}-3\right)\frac{\pi}{3} + {m'}_i^{(2)}\left(\frac{{n'}_i^{(2)}}{2}-3\right)\frac{\pi}{3}\right) \nonumber \\
                &+\frac{2\pi}{3}- \sum_i\sum_{0\leq j\leq2} {m'}_i^{(j)}\left(\frac{{n'}_i^{(j)}}{2}-3\right)\frac{\pi}{3}.
    \end{align}

    This finally means that: 
    \begin{align}
    \kappa_{\mathcal{D}}^v(g')+\kappa_{\mathcal{D}}^f(g')+ \sum_{\substack{c\in \mathcal{D}^{(2)}\\\text{triangle coming from }\\ \text{Lemma \ref{maximally positive} applied to $g'$}}} \kappa(c) -\frac{1}{2}\kappa_{\mathcal{D}}^f(g,g')
                 \leq \frac{2\pi}{3}.
    \label{curv est g'}
    \end{align}

    We could get the same estimate for $g$, but in this case we can do better. Again, denote by ${m}_i^{(j)}$-gons the number of ${n}_i^{(j)}$-gons, ${n}_i^{(j)}\geq 6$, that are $g$-cells with $j\in\{0,1,2\}$ $g$-adjacent triangles. Further, denote by $s$ the number of $g$-adjacent triangles.
    
    Let us first assume that $g$ has no $g$-maximally curved $n^{(0)}$-gons. The only thing that changes compared to the estimate for $g'$ is that at the first step instead of having $\sum_i {m'}_i^{(0)}\left(\frac{{n'}_i^{(0)}}{2}-1\right)\frac{\pi}{3}-(\sum_i{m'}_i^{(0)}-1)\frac{\pi}{3}$ we have $\sum_i {m'}_i^{(0)}\left(\frac{{n'}_i^{(0)}}{2}-2\right)\frac{\pi}{3}$.

    This means that entire estimate only changes for $\frac{\pi}{3}$, so the inequality bounding the curvature of $\mathcal{D}$ that comes from $g$ is: 
    \begin{align}
    \kappa_{\mathcal{D}}^v(g)+\kappa_{\mathcal{D}}^f(g) -\frac{1}{2}\kappa_{\mathcal{D}}^f(g,g')+ \sum_{\substack{c\in \mathcal{D}^{(2)}\\\text{triangle coming from }\\ \text{Lemma \ref{maximally positive} applied to $g$}}} \kappa(c)
                 \leq \frac{\pi}{3}.
    \label{curv est g}
    \end{align}

    Now assume that there are $g$-maximally positively curved $n^{(0)}$-gons. Using that $g$ is the lexicographically first geodesic to $w$, we will show that the bound above still needs to hold.

    From Lemma \ref{maximally positive} we get extra negative curvature between every two $g$-maximally positively curved $n^{(0)}$-gons. We will use the additional information we have about $g$ to show that there must also be extra negative curvature after the last $g$-maximally positively curved $n^{(0)}$-gon, which we will denote by $f_0$.

    Assume that there is no extra negative curvature after $f_0$ along $g$. Since $g$ is a combinatorial geodesic from $e$ and $f_0$ is $g$-maximally positively curved, the last vertex of $f_0$ on $g$, $v_0$, has to be the unique maximal vertex of $f_0$. The unique minimal vertex of $f_0$ also needs to be contained in $g$, which implies that all vertices of $f_0$ belong to $\overline{\cone{u_1}}$. 
    
    Because $g$ is a geodesic, the next vertex on it after $v_0$, $v_1$, has to be of strictly larger distance. In addition, the neighbouring vertex to $v_0$ that is not on $g$, $v_{-1}$, has to be of strictly smaller distance. These three vertices $v_{-1}$, $v_0$ and $v_1$ have to be on the next $g$-cell $f_1$, as otherwise $v_0$ would be contained in more than two cells, which would give extra negative curvature. 
    
    Now, $f_1$ needs to be almost $g$-maximally positively curved, but as $g$ enters $f_1$ from a non-minimal vertex, this implies that $f_1$ also has a unique maximal vertex which is on $g$. As $v_{-1}$ needs to be the unique minimal vertex of $f_1$, all vertices of $f_1$ also belong to $\overline{\cone{u_1}}$.
    
    We can continue this analysis along $g$ until we reach $w'$. Denote the $g$-almost maximally curved cell containing $w'$ by $f_{w'}$. All the vertices of $f_{w'}$ are in $\overline{\cone{u_1}}$. We already observed that $w'$ and $v'$ are of distance two, but Lemma \ref{comb_geod_enters} also tells us that the vertex between them, $v_{min}$ has to be the unique minimal vertex of $L^*(x)$ for some $x\in X^{(0)}$ containing $w'$ and $v'$. 
    
    If we assume that $\kappa(w')<\frac{\pi}{3}$, then $\kappa(w')\leq0$ and though we might have $\kappa_{\mathcal{D}}^v(g)+\kappa_{\mathcal{D}}^f(g) -\frac{1}{2}\kappa_{\mathcal{D}}^f(g,g')+ \sum_{\substack{c\in \mathcal{D}^{(2)}\\\text{triangle coming from }\\ \text{Lemma \ref{maximally positive} applied to $g$}}} \kappa(c) = \frac{2\pi}{3}$, we still get:
    
    \begin{equation}
        \kappa_{\mathcal{D}}^v(g)+\kappa_{\mathcal{D}}^f(g) -\frac{1}{2}\kappa_{\mathcal{D}}^f(g,g')+ \sum_{\substack{c\in \mathcal{D}^{(2)}\\\text{triangle coming from }\\ \text{Lemma \ref{maximally positive} applied to $g$}}} \kappa(c)+\kappa(w')
                 \leq \frac{2\pi}{3}.
    \label{final estimate for g}
    \end{equation}
    
    Otherwise, if $\kappa(w')=\frac{\pi}{3}$, since there is no extra negative curvature along $g$, the vertex $v_{min}$ is also a part of $f_{w'}$, so it is in $\overline{\cone{u_1}}$. But this would mean that $v'$ is also in $\overline{\cone{u_1}}$, which is a contradiction.

    To conclude, the estimate \ref{final estimate for g} always holds, either because $g$ has smaller curvature than a regular combinatorial geodesic, or because $w'$ is negatively curved.
    
    Incorporating the curvature estimates \ref{final estimate for g} and \ref{curv est g'} for $g$ and $g'$ into the inequality \ref{lex_curv_est2} (along with $\kappa(v')\leq\frac{\pi}{3}$) we get:

    \begin{equation}
        0\leq\kappa_{\mathcal{D}}^v(l)
        \label{lex_curv_est3}
    \end{equation}

    However, the disc diagram $\mathcal{D}$ at $v_{min}$ has an interior angle $\frac{4\pi}{3}$, so we get a contradictory upper bound:
    
    \begin{equation}
        \kappa_{\mathcal{D}}^v(l)=\kappa(v_{min})\leq-\frac{\pi}{3}.
    \end{equation}
    
\end{proof}

\bibliographystyle{plain}

\end{document}